\documentclass{amsart}
\usepackage{xy}
\usepackage{accents}
\usepackage{breqn}
\usepackage{sseq}
\usepackage{amssymb}
\usepackage{graphicx}
\usepackage{rotating,multicol,afterpage}
\usepackage{comment}

\usepackage[foot]{amsaddr}
\xyoption{all}
\title{The Hochschild homology and cohomology of $A(1)$.}
\author{A. Salch}
\address{Department of Mathematics, Wayne State University, Detroit, MI, USA}
\email{asalch@wayne.edu}
\usepackage{hyperref}
\newcommand{\dbtilde}[1]{\widetilde{\raisebox{0pt}[0.85\height]{$\widetilde{#1}$}}}

\theoremstyle{plain}
\newtheorem{prop}{Proposition}[section]
\newtheorem{theorem}[prop]{Theorem}
\newtheorem{corollary}[prop]{Corollary}
\newtheorem{setup for thm}[prop]{Setup for theorem}
\newtheorem{necessary condition}[prop]{Necessary condition}

\newtheorem*{unnumberedtheorem}{Theorem}
\newtheorem*{unn motivating questions}{Motivating Questions}
\newtheorem{lemma}[prop]{Lemma}

\newtheorem{motivating questions}[prop]{Motivating Questions}
\newtheorem{definition}[prop]{Definition}
\newtheorem{definition-proposition}[prop]{Definition-Proposition}
\newtheorem{definition-theorem}[prop]{Definition-Theorem}

\newtheorem{running notations}[prop]{Running notations}

\theoremstyle{definition}

\newtheorem{remark}[prop]{Remark}
\newtheorem{intuitive idea}[prop]{Intuitive idea}
\newtheorem*{history-and-status-of-this-paper}{History and status of this paper}

\newtheorem{observation}[prop]{Observation}

\newtheorem{convention}[prop]{Conventions}

\DeclareMathAlphabet{\mathcal}{OT1}{pzc}{m}{it}

\DeclareMathOperator{\BAR}{{\rm Bar}}

\DeclareMathOperator{\Sq}{{\rm Sq}}

\DeclareMathOperator{\op}{{\rm op}}

\DeclareMathOperator{\Ext}{{\rm Ext}}

\DeclareMathOperator{\Tor}{{\rm Tor}}

\usepackage{cleveref}

\begin{document}

\begin{abstract}
We compute the Hochschild homology and cohomology of $A(1)$, the subalgebra of the $2$-primary Steenrod algebra generated by the first two Steenrod squares,
$\Sq^1,\Sq^2$. The computation is accomplished using several May-type spectral sequences.
\end{abstract}

\maketitle

\section{Introduction.}
The $2$-primary Steenrod algebra $A$, that is, the algebra of stable natural endomorphisms of the mod $2$ cohomology functor on topological spaces, 
has generators $\Sq^1,\Sq^2,\Sq^3, \dots$, the {\em Steenrod squares.} The subalgebra of $A$ generated by the first two Steenrod squares,
$\Sq^1$ and $\Sq^2$, is called $A(1)$. The $\mathbb{F}_2$-algebra $A(1)$ is an eight-dimensional, graded, noncommutative, co-commutative Hopf algebra. 
The homological algebra of $A(1)$-modules effectively determines, via the Adams spectral sequence, 
the $2$-complete homotopy theory of spaces and spectra
smashed with the connective real $K$-theory spectrum $ko$. These ideas are all classical; an excellent reference for the Steenrod algebra is Steenrod's book
\cite{MR0145525}, and an excellent reference for $A(1)$-modules and the Adams spectral sequence is the third chapter of Ravenel's book~\cite{MR860042}.

The algebra $A(1)$ is well-used by homotopy theorists, but since the audience for this paper may include algebraists who have never worked with $A(1)$, we give a description of its structure. A minimal set of generators for $A(1)$, as a graded $\mathbb{F}_2$-algebra, is $\{ \Sq^1,\Sq^2\}$, with the former in degree $1$ and the latter in degree $2$. A minimal set of relations between these generators is $\Sq^1\Sq^1=0$ and $\Sq^2\Sq^2 = \Sq^1\Sq^2\Sq^1$. The coproduct on $A(1)$ is determined by the fact that $\Sq^1$ is primitive and $\Delta(\Sq^2) = \Sq^2\otimes 1 + \Sq^1 \otimes \Sq^1 + 1\otimes \Sq^2$. The augmentation map $\epsilon: A(1)\rightarrow \mathbb{F}_2$ is determined by the grading degrees and the fact that $A(1)$ is a graded Hopf algebra, i.e., $\epsilon(\Sq^1) = 0 = \epsilon(\Sq^2)$.

Among homotopy theorists, the algebra $A(1)$ is a venerable and standard ``test case'' for any new construction which is applicable to a noncommutative ring. Computing $HH_*(A(1),A(1))$, however, is a nontrivial task, and it seems that before this paper, this computation has 
never been successfully done. B\"{o}kstedt, in his extremely influential unpublished paper on topological Hochschild homology, computes
the Hochschild homology of $\pi_*(H\mathbb{F}_p\smash H\mathbb{F}_p)$, i.e., the Hochschild homology of {\em the linear dual} of the entire Steenrod algebra,
but this is very straightforward, since the dual of the Steenrod algebra is polynomial at $p=2$ and polynomial tensored with exterior at $p>2$.
For the same reason, it is also easy to compute the Hochschild homology of {\em the linear dual} of $A(1)$. But the linear dual of $A(1)$ is a completely different ring than $A(1)$, so this sheds no light on 
the Hochschild homology of $A(1)$ itself! The reader who might hope that the Hochschild homology of a finite-dimensional Hopf algebra $R$ might be linearly dual to the Hochschild homology of the dual Hopf algebra can consider the counterexample $R = k[x]/x^4$ with $k$ a field of characteristic two and $x$ primitive.

In this paper we compute $HH_*(A(1),A(1))$ by using two different filtrations on $A(1)$ and studying the spectral sequences in Hochschild homology arising from
these filtrations. These spectral sequences are the analogues in Hochschild homology of J. P. May's spectral sequence \cite{MR2614527} for computing $\Ext$ over the Steenrod algebra, so we think of these as ``May-type'' spectral sequences. 

Computing $HH_*(A(1),A(1))$ takes some work, since $A(1)$ is noncommutative, so $HH_*(A(1),A(1))$ does not have a natural product structure. As a consequence, the May-type spectral sequence converging to $HH_*(A(1),A(1))$ that one would 
construct in the most na\"{i}ve way is not {\em multiplicative}, i.e., it 
does not have a product satisfying a Leibniz rule. This makes the computation of differentials in that spectral sequence
basically intractable. Instead, we take the linear dual of the standard Hochschild chain complex on $A(1)$, and we use the co-commutative coproduct on $A(1)$
to give the cohomology of this linear dual cochain complex a product structure arising from the coproduct on $A(1)$ and the linear dual of the Alexander-Whitney map. As a consequence, while the product on the $\mathbb{F}_2$-algebra $A(1)$ is used to define the Hochschild homology $HH_*(A(1),A(1))$ as a $\mathbb{F}_2$-vector space, it is the coproduct on $A(1)$ which gives us the useful ring structure on $HH_*(A(1),A(1))^*$. This is in contrast to the more familiar situation of a commutative $k$-algebra $R$: it is the product on $R$ which is used both to define the $k$-vector space $HH_*(R,R)$ and to equip it with a product.

In Proposition~\ref{SSs and cocycle reps} we set up {\em multiplicative} spectral sequences
computing the cohomology of the linear dual cochain complex of the standard Hochschild chain complex of $A(1)$.
By an easy universal coefficient theorem argument (Proposition~\ref{duality between hh and cohh}), 
this cohomology is the $\mathbb{F}_2$-linear dual of the desired Hochschild homology $HH_*(A(1),A(1))$.

We then compute the differentials in these spectral sequences. In the end there are nonzero $d_1$ and $d_2$ differentials, and no nonzero differentials on any later terms of the spectral sequences. In~\ref{ss e2} and~\ref{ss e3} we present charts of the $E_2$ and $E_3\cong E_{\infty}$-pages of the the relevant spectral sequences.
Our charts are drawn using the usual Adams spectral sequence conventions, described below.
This is the best format if, for example, one wants to use this Hochschild homology as the input for an Adams spectral sequence,
and it also makes it easier to see the natural map from this Hochschild homology to the classical Adams spectral sequence computing $\pi_*(ko)^{\widehat{}}_2$, the $2$-complete
homotopy groups of the connective real $K$-theory spectrum $ko$, in Proposition~\ref{ss comparison with supertrivial coeffs} and in the charts~\ref{ss e3} and~\ref{classical may e3}.

In particular, the chart~\ref{ss e3} is a chart of the ($\mathbb{F}_2$-linear dual of the) Hochschild homology of $A(1)$, and gives our most detailed description
of $HH_*(A(1),A(1))$. We reproduce that chart here:
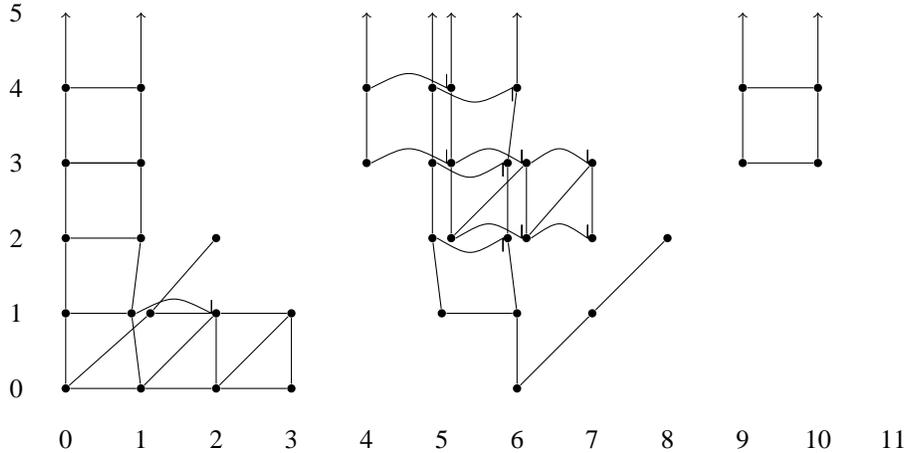
\begin{equation*}\begin{sseq}[grid=none,entrysize=10mm,labelstep=1]{0...11}{0...5}
 \ssmoveto 0 0 
 \ssdropbull
 \ssname{1}

 \ssmove 1 0
 \ssdropbull 
 \ssname{x10} 

 \ssmove 1 0
 \ssdropbull
 \ssname{x11}

 \ssmove 1 0
 \ssdropbull
 \ssname{x10x11}

 \ssmove 3 0
 \ssdropbull
 \ssname{x6}

 \ssmove{-6}{1} 
 \ssdropbull
 \ssname{h10}

 \ssmove 1 0
 \ssdropbull 
 \ssname{x10h10} 
 \ssdropbull
 \ssname{h11}

 \ssmove 1 0
 \ssdropbull
 \ssname{x11h10}

 \ssmove 1 0
 \ssdropbull
 \ssname{x11h11}

 \ssmove 2 0
 \ssdropbull
 \ssname{z}

 \ssmove 1 0
 \ssdropbull
 \ssname{h10x6}

 \ssmove 1 0
 \ssdropbull
 \ssname{h11x6}

 \ssmove{-7}{1}
 \ssdropbull
 \ssname{h10^2}
 
 \ssmove 1 0
 \ssdropbull
 \ssname{x10h10^2}

 \ssmove 1 0
 \ssdropbull
 \ssname{h11^2}

 \ssmove 3 0
 \ssdropbull
 \ssname{h10z}
 \ssdropbull
 \ssname{x10b20}

 \ssmove 1 0
 \ssdropbull
 \ssname{h10^2x6}
 \ssdropbull
 \ssname{x10^2b20}

 \ssmove 1 0
 \ssdropbull
 \ssname{x10^3b20}

 \ssmove 1 0
 \ssdropbull
 \ssname{h11^2x6}


 \ssmove{-8}{1}
 \ssdropbull
 \ssname{h10^3}

 \ssmove 1 0
 \ssdropbull
 \ssname{x10h10^3}

 \ssmove 3 0
 \ssdropbull
 \ssname{h10b20}

 \ssmove 1 0 
 \ssdropbull
 \ssname{h10^2z}
 \ssdropbull
 \ssname{x10h10b20}

 \ssmove 1 0
 \ssdropbull
 \ssname{h10^3x6}
 \ssdropbull
 \ssname{x10^2h10b20}

 \ssmove 1 0
 \ssdropbull
 \ssname{x10^3h10b20}

 \ssmove 2 0
 \ssdropbull
 \ssname{h10b20z}

 \ssmove 1 0
 \ssdropbull
 \ssname{x10h10b20z}

 \ssmove{-10}{1}
 \ssdropbull
 \ssname{h10^4}

 \ssmove 1 0
 \ssdropbull
 \ssname{x10h10^4}

 \ssmove 3 0
 \ssdropbull
 \ssname{h10^2b20}

 \ssmove 1 0
 \ssdropbull
 \ssname{h10^3z}
 \ssdropbull
 \ssname{x10h10^2b20}

 \ssmove 1 0
 \ssdropbull
 \ssname{h10^4x6}

 \ssmove 3 0
 \ssdropbull
 \ssname{h10^2b20z}

 \ssmove 1 0
 \ssdropbull
 \ssname{x10h10^2b20z}

 \ssgoto{1}
 \ssgoto{x10}
 \ssstroke
 \ssgoto{x10}
 \ssgoto{x11}
 \ssstroke
 \ssgoto{x11}
 \ssgoto{x10x11}
 \ssstroke
 \ssgoto{1}
 \ssgoto{h10}
 \ssstroke
 \ssgoto{1}
 \ssgoto{h11}
 \ssstroke

 \ssgoto{x10}
 \ssgoto{x10h10}
 \ssstroke
 \ssgoto{x10}
 \ssgoto{x11h10}
 \ssstroke

 \ssgoto{x11}
 \ssgoto{x11h10}
 \ssstroke
 \ssgoto{x11}
 \ssgoto{x11h11}
 \ssstroke
 \ssgoto{x10x11}
 \ssgoto{x11h11}
 \ssstroke

 \ssgoto{x6}
 \ssgoto{h10x6}
 \ssstroke
 \ssgoto{x6}
 \ssgoto{h11x6}
 \ssstroke
 \ssgoto{h11x6}
 \ssgoto{h11^2x6}
 \ssstroke

 \ssgoto{h10}
 \ssgoto{h10^2}
 \ssstroke
 \ssgoto{h10}
 \ssgoto{x10h10}
 \ssstroke
 \ssgoto{x10h10}
 \ssgoto{x11h10}
 \ssstroke[curve=.1]

 \ssgoto{x10h10}
 \ssgoto{x10h10^2}
 \ssstroke

 \ssgoto{h11}
 \ssgoto{x11h10}
 \ssstroke
 \ssgoto{x11h10}
 \ssgoto{x11h11}
 \ssstroke
 \ssgoto{h11}
 \ssgoto{h11^2}
 \ssstroke

 \ssgoto{h10^2}
 \ssgoto{h10^3}
 \ssstroke
 \ssgoto{h10^3}
 \ssgoto{h10^4}
 \ssstroke
 \ssgoto{h10^2}
 \ssgoto{x10h10^2}
 \ssstroke
 \ssgoto{h10^3}
 \ssgoto{x10h10^3}
 \ssstroke

 \ssgoto{x10h10^2}
 \ssgoto{x10h10^3}
 \ssstroke
 \ssgoto{x10h10^3}
 \ssgoto{x10h10^4}
 \ssstroke

 \ssgoto{h10^3}
 \ssgoto{x10h10^3}
 \ssstroke
 \ssgoto{h10^4}
 \ssgoto{x10h10^4}
 \ssstroke

 \ssgoto{z}
 \ssgoto{h10x6}
 \ssstroke
 \ssgoto{h10z}
 \ssgoto{h10^2x6}
 \ssstroke[curve=-.1]
 \ssgoto{h10^2z}
 \ssgoto{h10^3x6}
 \ssstroke[curve=-.1]
 \ssgoto{h10^3z}
 \ssgoto{h10^4x6}
 \ssstroke[curve=-.1]

 \ssgoto{z}
 \ssgoto{h10z}
 \ssstroke
 \ssgoto{h10z}
 \ssgoto{h10^2z}
 \ssstroke
 \ssgoto{h10^2z}
 \ssgoto{h10^3z}
 \ssstroke

 \ssgoto{h10x6}
 \ssgoto{h10^2x6}
 \ssstroke
 \ssgoto{h10^2x6}
 \ssgoto{h10^3x6}
 \ssstroke
 \ssgoto{h10^3x6}
 \ssgoto{h10^4x6}
 \ssstroke

 \ssgoto{h10b20}
 \ssgoto{x10h10b20}
 \ssstroke[curve=.1]
 \ssgoto{x10h10b20}
 \ssgoto{x10^2h10b20}
 \ssstroke[curve=.1]
 \ssgoto{x10^2h10b20}
 \ssgoto{x10^3h10b20}
 \ssstroke
 \ssgoto{x10b20}
 \ssgoto{x10^2b20}
 \ssstroke[curve=.1]
 \ssgoto{x10^2b20}
 \ssgoto{x10^3b20}
 \ssstroke
 \ssgoto{h10^2b20}
 \ssgoto{x10h10^2b20}
 \ssstroke[curve=.1]
 \ssgoto{h10b20z}
 \ssgoto{x10h10b20z}
 \ssstroke
 \ssgoto{h10^2b20z}
 \ssgoto{x10h10^2b20z}
 \ssstroke

 \ssgoto{h10b20}
 \ssgoto{h10^2b20}
 \ssstroke
 \ssgoto{x10h10b20}
 \ssgoto{x10h10^2b20}
 \ssstroke
 \ssgoto{x10b20}
 \ssgoto{x10h10b20}
 \ssstroke
 \ssgoto{x10^2b20}
 \ssgoto{x10^2h10b20}
 \ssstroke
 \ssgoto{x10^3b20}
 \ssgoto{x10^3h10b20}
 \ssstroke
 \ssgoto{h10b20z}
 \ssgoto{h10^2b20z}
 \ssstroke
 \ssgoto{x10h10b20z}
 \ssgoto{x10h10^2b20z}
 \ssstroke

 \ssgoto{x10b20}
 \ssgoto{x10^2h10b20}
 \ssstroke
 \ssgoto{x10^2b20}
 \ssgoto{x10^3h10b20}
 \ssstroke

 \ssgoto{h10^4}
 \ssmove 0 1
 \ssstroke[void,arrowto]
 \ssgoto{x10h10^4}
 \ssmove 0 1
 \ssstroke[void,arrowto]

 \ssgoto{h10^3z}
 \ssmove 0 1
 \ssstroke[void,arrowto]
 \ssgoto{h10^4x6}
 \ssmove 0 1
 \ssstroke[void,arrowto]
 \ssgoto{h10^2b20}
 \ssmove 0 1
 \ssstroke[void,arrowto]
 \ssgoto{x10h10^2b20}
 \ssmove 0 1
 \ssstroke[void,arrowto]
 \ssgoto{h10^2b20z}
 \ssmove 0 1
 \ssstroke[void,arrowto]
 \ssgoto{x10h10^2b20z}
 \ssmove 0 1
 \ssstroke[void,arrowto]

\end{sseq}\end{equation*}
The vertical axis is homological degree, so the row $s$ rows above the bottom of the chart is the associated graded $\mathbb{F}_2$-vector space of a filtration
on $HH_s(A(1),A(1))$. The horizontal axis is, following the tradition in homotopy theory, the Adams degree, i.e., the topological degree $u$ (coming from the topological grading on $A(1)$) minus the homological degree $s$. 
The horizontal lines in the chart describe comultiplications by certain elements in the linear dual Hopf algebra $\hom_{\mathbb{F}_2}(A(1),\mathbb{F}_2)$ of $A(1)$,
and the nonhorizontal lines describe certain operations in the linear dual of $HH_*(A(1),A(1))$, described in Convention \ref{conventions on ss charts}.
The entire pattern described by this chart is repeated every four vertical degrees and every eight horizontal degrees: there is a periodicity class (not pictured) in bidegree $(s=4,u-s=8)$.

Information about the $\mathbb{F}_2$-vector space dimension of $HH_*(A(1),A(1))$ in each grading degree is provided by Theorem~\ref{e-p series thm},
which we reproduce below. We do not describe any ring structure on $HH_*(A(1),A(1))$ because $A(1)$ is noncommutative and so there is no natural ring structure on its Hochschild homology.
\begin{unnumberedtheorem}
The $\mathbb{F}_2$-vector space dimension of $HH_n(A(1),A(1))$ is:
\[ \dim_{\mathbb{F}_2} HH_n(A(1),A(1)) = \left\{ 
    \begin{array}{ll} 
      2n+5 &\mbox{\ if\ } 2\mid n \\
      2n+6 &\mbox{\ if\ } n\equiv 1 \mod 4 \\
      2n+4 &\mbox{\ if\ } n\equiv 3 \mod 4 .\end{array}\right. \]
Hence the Poincar\'{e} series of the graded $\mathbb{F}_2$-vector space $HH_*(A(1),A(1))$ is
\[ \frac{ 5 + 8s + 9s^2 + 10s^3 + \frac{8s^4}{1-s}}{1-s^4} .\]

If we additionally keep track of the extra grading on $HH_*(A(1),A(1))$ coming from the topological grading on $A(1)$, then 
the Poincar\'{e} series of the bigraded $\mathbb{F}_2$-vector space $HH_{*,*}(A(1),A(1))$ is
\begin{dmath*} \left( (1+u)\left(1+u^2 + su(1+u^2+u^5) + s^2u^2(1+2u^5+u^7) + s^3u^3(1+u^4+u^5+u^6+u^9) + \frac{s^4u^4(1+u^4+u^5+u^9)}{1-su}\right) + u^6+su^2+su^8+s^2u^4\right)\frac{1}{1-s^4u^{12}}\end{dmath*}
where $s$ is the homological degree and $u$ is the topological degree.
\end{unnumberedtheorem}

Our computation of $HH_*(A(1),A(1))$ can be used as the input for other spectral sequences in order to make further computations. For example,
one could use it as input for the Connes spectral sequence, as in~9.8.6 of~\cite{MR1269324}, computing the cyclic homology $HC_*(A(1))$, or as input for 
the Pirashvili--Waldhausen spectral sequence, as in \cite{MR1181095}, computing the topological Hochschild
homology $THH_*(A(1))$. Those $THH$-groups are the input one would use to run the necessary homotopy fixed-point spectral sequences to compute the topological cyclic homology $TC_*(A(1))$,
which, using McCarthy's theorem (see~\cite{MR3013261}), 
gives the $2$-complete algebraic $K$-groups $K_*(A(1))^{\widehat{}}_{2}$. 
See \cite{MR1474979} for a survey of trace method computations of this kind.
Those computations are entirely outside the scope of the present paper.

We prefer to give much simpler applications: in \cref{The Hochschild cohomology...}, we use our Hochschild homology calculations together with a well-known duality argument to calculate the Hochschild cohomology groups $HH^*(A(1),A(1))$. Using the relationship between graded deformations of algebras and graded Hochschild cohomology from \cite{MR1383469}, we give some deformation-theoretic consequences in Corollary \ref{uniqueness cor}: for example, we find that there are precisely four isomorphism classes of first-order graded deformations of the $\mathbb{F}_2$-algebra $A(1)$.

We remark that the methods used in this paper also admit basically obvious extensions to methods for computing $HH_*(A(n),A(n))$ for arbitrary $n$, but one sees that for $n>1$,
carrying out such computations would be a daunting task. Our HH-May spectral sequence of Proposition~\ref{SSs and cocycle reps}
surjects on to the classical May spectral sequence computing $\Ext_{A(1)}^*(\mathbb{F}_2,\mathbb{F}_2)$, 
and for the same reasons, the $n>1$ analogue of our HH-May spectral sequence maps naturally to 
the classical May spectral sequence computing $\Ext_{A(n)}^*(\mathbb{F}_2,\mathbb{F}_2)$. We suspect that this map is still surjective for $n>1$, 
although we have made no attempt to verify this. Consequently the computation of $HH_*(A(n),A(n))$ using our methods is of at least the same
level of difficulty as the computation of $\Ext_{A(n)}^*(\mathbb{F}_2,\mathbb{F}_2)$. For $n=2$ this is already quite nontrivial.

There have been many cases of spectral sequence calculations of Hochschild (co) homology. We single out one particularly interesting precedent for this paper: in \cite{MR950556}, J.-L. Brylinski constructs a spectral sequence related to the present paper's abelianizing spectral sequence. Brylinski's spectral sequence computes the Hochschild homology of a noncommutative algebra over a field of characteristic zero, for the purposes of studying Poisson manifolds, and remarks that ``[e]xamples show that this spectral sequence tends to degenerate at $E^2$,'' in particular, that Brylinski's spectral sequence
collapses at the $E^2$-term for the algebra of differential operators on an algebraic or complex-analytic manifold. In the present paper's computation of $HH_*(A(1), A(1))$ we instead get collapse one term later, at $E_3$ rather than $E_2$. 

\begin{remark}
Given a field $k$ and a graded $k$-algebra $A$, there are {\em two} notions of the Hochschild homology $HH_*(A,A)$ which are in common circulation: 
\begin{itemize}
\item one can forget the grading, and simply consider the Hochschild homology of the underlying ungraded $k$-algebra of $A$.
This is the right thing to do in many applications of Hochschild homology in classical algebra; for example, if one wants to use trace methods to compute the algebraic $K$-groups of (the underlying ungraded algebra of) $A$. The computations in~\cite{MR972360} are an excellent example of this.
\item Alternatively, one can instead compute the ``graded algebra Hochschild homology,'' which incorporates a sign convention into the cyclic bar complex. 
This is the right thing to do in many applications of Hochschild homology in algebraic topology; for example, if one wants to use B\"{o}kstedt's spectral sequence to compute topological Hochschild homology of a ring spectrum. The computations in~\cite{MR1209233} are an excellent example of this.
\end{itemize}
Since $A(1)$ has characteristic $2$, sign conventions are irrelevant, and both notions of Hochschild homology coincide. This makes the computations in this paper equally applicable in classical algebra as in algebraic topology. Any future odd-primary analogues of these computations, however, would require that one carry out ``ungraded'' $HH_*$ computations separately from the ``graded'' $HH_*$ computations.
\end{remark}

We are grateful to an anonymous referee for helpful suggestions.

\section{Construction of May-type spectral sequences for Hochschild homology.}

The spectral sequence of Proposition \ref{may ss for hh} is classical, but in what follows, we will rely on specific structure and properties of this spectral sequence. We prefer to give a self-contained treatment of the spectral sequence here, to fix our conventions about indexing in the spectral sequence, and also so that the reader can easily see that our claims about the structure and properties of the spectral sequence are indeed true. 
In the statement of Theorem \ref{may ss for hh}, the bigrading subscripts 
$HH_{s,t}$ are as follows: $s$ is the usual homological
degree, while $t$ is the filtration degree, defined and computed as follows: 
given a homology class
$x\in HH_s(E^0A,E^0A)$, its filtration degree is the total degree (in the grading
on $E^0A$ induced by the filtration on $A$) of any homogeneous 
cycle representative for $x$ in the standard Hochschild chain complex.
\begin{prop}{\bf (May spectral sequence for Hochschild homology.)}\label{may ss for hh}
Let $k$ be a field, $A$ an algebra, and 
\begin{equation}\label{filt 0} A = F^0A \supseteq F^1A \supseteq F^2A \supseteq \dots \end{equation}
a filtration of $A$ which is multiplicative, that is,
if $x\in F^mA$ and $y\in F^nA$, then $xy\in F^{m+n}A$.
Then there exists a spectral sequence
\begin{align*} 
 E_1^{s,t} \cong HH_{s,t}(E_0A, E_0A) & \Rightarrow  HH_s(A,A) \\
 d_r^{s,t} : E_r^{s,t} & \rightarrow E_r^{s-1,t+r}.
 \end{align*}

This spectral sequence enjoys the following properties:
\begin{itemize}
\item 
If the filtration~\ref{filt 0} is finite, 
then the spectral sequence converges strongly.
\item
If $A$ is also a graded $k$-algebra and the filtration layers
$F^nA$ are graded sub-$k$-modules of $A$,
then the differential in the spectral sequence preserves the grading.
\item 
If $A$ is commutative, then so is $E_0A$, and the spectral sequence is one of algebras\footnote{That is, the differentials in the spectral sequence obey the graded Leibniz rule. The ring structures on the input $HH_*(E_0A, E_0A)$ and abutment $HH_*(A,A)$ are given by the usual shuffle product on the Hochschild homology of a commutative ring.}. Furthermore, the product in the spectral sequence converges to the product induced on the associated graded by the usual
shuffle product on $HH_*(A,A)$. 
\end{itemize}
\end{prop}
\begin{proof}
Let $CC_{\bullet}(A,A)$ denote the standard Hochschild chain complex of 
$A$, and let $F^nCC_{\bullet}(A,A)$ denote the sub-chain-complex of
$CC_{\bullet}(A,A)$ consisting of all chains of total filtration degree
$\leq n$. 
Our May spectral sequence is now simply the spectral sequence of the
filtered chain complex
\begin{equation}\label{filt 1} CC_{\bullet}(A,A) = F^0CC_{\bullet}(A,A) \supseteq F^1CC_{\bullet}(A,A)
 \supseteq F^2CC_{\bullet}(A,A)\supseteq \dots .\end{equation}
If $A$ is commutative, then filtration~\ref{filt 0} being multiplicative
implies that filtration~\ref{filt 1} is a multiplicative filtration
of the differential graded algebra $CC_{\bullet}(A,A)$, with product
given by the shuffle product. See \cite{MR1269324} for a textbook treatment of the shuffle product. It is standard (e.g. see \cite{MR0060829}) that the spectral sequence
of a multiplicatively-filtered DGA is multiplicative.

The product on the spectral sequence being given by the shuffle product is due to the naturality of the construction of $CC_{\bullet}(A,A)$ in the choice of $k$-algebra $A$: if $A$ is commutative, then the multiplication map $A\otimes_k A \rightarrow A$ is a morphism of $k$-algebras, hence we get a map of chain complexes
\[ CC_{\bullet}(A\otimes_k A, A\otimes_k A) \rightarrow CC_{\bullet}(A,A),\] which we compose with the Eilenberg-Zilber (i.e., ``shuffle'') isomorphism
\[ CC_{\bullet}(A,A) \otimes_k CC_{\bullet}(A,A) \stackrel{\cong}{\longrightarrow} CC_{\bullet}(A\otimes_k A, A\otimes_k A) .\]
\end{proof}

\begin{remark}\label{how to compute differentials}
The differential in the spectral sequence of Proposition \ref{may ss for hh} is (like any other spectral sequence
of a filtered chain complex) computed on a class $x\in HH_{*,*}(E_0A,E_0A)$
by computing a homogeneous cycle representative $y$ for $x$
in the standard Hochschild chain complex for $E_0A$, lifting
$y$ to a homogeneous chain $\widetilde{y}$ in the standard Hochschild
chain complex for $A$, applying the Hochschild differential $d$ to
$\widetilde{y}$, then taking the image of $d\widetilde{y}$ in 
the standard Hochschild chain complex for $E_0A$.
\end{remark}

The following construction is classical: as far as the author knows, it was first introduced in \cite{cartiercohomologie} (see also section 3.1 of \cite{MR597479}):
\begin{definition}\label{def of cyclic cobar}
Let $k$ be a field and $C$ a coalgebra over $k$ with comultiplication map $\Delta: C \rightarrow C\otimes_k C$. By the {\em Cartier cobar construction on $C$} we mean the cosimplicial $k$-vector space
\[ \xymatrix{ 
C \ar@<1ex>[r] \ar@<-1ex>[r] & 
 C\otimes_k C  \ar[l]\ar@<2ex>[r]\ar[r]\ar@<-2ex>[r] & 
 C\otimes_k C\otimes_k C  \ar@<1ex>[l] \ar@<-1ex>[l] \ar@<3ex>[r] \ar@<1ex>[r] \ar@<-1ex>[r] \ar@<-3ex>[r] & \dots \ar@<2ex>[l]\ar[l]\ar@<-2ex>[l]  }\]
with coface maps $d^0, d^1,\dots ,d^n: C^{\otimes_k n} \rightarrow C^{\otimes_k n+1}$ given by
\begin{align*} d^i(a_0\otimes \dots \otimes a_{n-1}) &= 
  \left\{ \begin{array}{l} a_0 \otimes \dots \otimes a_{i-1}\otimes \Delta(a_i)\otimes a_{i+1} \otimes \dots \otimes a_{n-1} \\
                             \mbox{if\ } i < n,\\
                            \tau\left( \Delta(a_0) \otimes a_1 \otimes \dots \otimes a_{n-1}\right) \\
                             \mbox{if\ } i = n,\end{array}\right. \end{align*}
where $\tau$ is the cyclic permutation toward the left, i.e., 
\[\tau( a_0 \otimes \dots \otimes a_n) = a_1 \otimes \dots \otimes a_n\otimes a_0.\]
The codegeneracy maps are constructed from the counit (augmentation) map on $C$ in the usual way.

By the {\em Cartier cochain complex of $C$,} denoted $CartC^{\bullet}(C,C)$, we mean the
alternating sum cochain complex of the Cartier cobar construction on $C$.
We will write $CartH^*(C,C)$ for its cohomology, which we call {\em Cartier cohomology.}
\end{definition}
\begin{remark}\label{explicit coproduct remark}
Clearly, if $A$ is a finite-dimensional co-commutative Hopf algebra over a field $k$, and if $C$ is the $k$-linear dual of $A$, then the Cartier cochain complex on $C$ is isomorphic to the $k$-linear dual of the Hochschild chain complex on $A$, and the Cartier cochain complex then inherits a product from the coproduct on the Hochschild chain complex of $A$. 

It is worthwhile to be very explicit about the coproduct on the Hochschild chain complex of a co-commutative Hopf algebra, since this point is not as well documented in the literature as it could be (although it does appear in some places, for example, in \cite{MR3519050}). A Hochschild $n$-chain is an element of $A\otimes_k A\otimes_k \dots \otimes_k A$, with $n+1$ tensor factors of $A$; because the tensor factor on the far left plays a special role, we will adopt the common notation $a_0\left[ a_1 \otimes \dots \otimes a_n\right]$ instead of $a_0\otimes a_1\otimes \dots\otimes a_n$ for a (term in a) Hochschild $n$-chain.

Then the coproduct on the Hochschild chain complex is given by
\[ \Delta\left( a_0\left[ a_1 \otimes \dots \otimes a_n\right]\right) 
  = \sum_{i=1}^{\ell} \sum_{j=0}^n a_{0,i}^{\prime} \left[ a_1\otimes \dots \otimes a_j\right] \otimes a_{0,i}^{\prime\prime} \left[ a_{j+1}\otimes \dots \otimes a_n\right] \]
where $\ell$ and $\{ a_{0,i}^{\prime},a_{0,i}^{\prime\prime}\}_{i=1, \dots ,\ell}$ are given by the determined by the coproduct in $A$ and the formula
\[ \Delta(a_0) = \sum_{i=0}^{\ell} a_{0,i}^{\prime}\otimes a_{0,i}^{\prime\prime}.\]
\end{remark}

It is not difficult to see how to adapt Definition \ref{def of cyclic cobar} to yield the Cartier cochain complex with coefficients in an $C$-bicomodule $M$. An explicit account is given in \cite{MR597479}. This of course also yields a definition of $CartH^*(C,M)$. 

Proposition \ref{duality between hh and cohh} is straightforward and certainly not new:
\begin{prop}\label{duality between hh and cohh}
Let $k$ be a field and let $A$ be a $k$-algebra which is finite-dimensional as a $k$-vector space.
Let $A^*$ denote the $k$-linear dual coalgebra of $A$.
Then, for each $n\in\mathbb{N}$, the $n$th Hochschild homology $k$-vector space of $A$ and the $n$th Cartier cohomology $k$-vector space
of $A^*$ are mutually $k$-linearly dual. That is, we have isomorphisms of $k$-vector spaces:
\begin{align}
\label{iso 6} \hom_k( HH_n(A,A), k) &\cong CartH^n(A^*,A^*), \\
\label{iso 7} \hom_k( CartH^n(A^*,A^*), k) &\cong HH_n(A,A),\end{align}
as well as isomorphisms
\begin{align}
\label{iso 8} \hom_k( HH_n(A,k), k) &\cong CartH^n(A^*,k), \\
\label{iso 9} \hom_k( CartH^n(A^*,k), k) &\cong HH_n(A,k).\end{align}
\end{prop}
\begin{proof}
By construction, the Cartier cobar construction is simply the $k$-linear dual of the usual Hochschild bar construction, so isomorphisms \eqref{iso 6} through \eqref{iso 9} simply follow from the universal coefficient theorem.
\end{proof}

In Proposition \ref{may ss for dual hh}, the bigrading superscripts $CartH^{s,t}$ are as follows: $s$ is the usual cohomological degree, while $t$ is the filtration degree, defined and computed as follows using cocycle representatives in the same way as the filtration degree in the spectral sequence of Proposition \ref{may ss for hh}.
\begin{prop}{\bf (May spectral sequence for dual Hochschild cohomology.)}\label{may ss for dual hh}
Let $k$ be a field and let $C$ be a $k$-coalgebra.
Let 
\begin{equation}\label{filt 0 dual} F_0C \subseteq F_1C \subseteq F_2C \subseteq \dots \end{equation}
be a filtration of $C$ which is comultiplicative, that is,
if $x\in F_mC$, then $\Delta(x)\in \sum_{n=0}^m F_{n}C\otimes F_{m-n}C$.
Then there exists a spectral sequence
\begin{align*} 
 E_1^{s,t} \cong CartH^{s,t}(E^0C, E^0C) & \Rightarrow  CartH^s(C,C) \\
 d_r^{s,t} : E_r^{s,t} & \rightarrow E_r^{s+1,t-r}.
 \end{align*}

This spectral sequence enjoys the following additional properties:
\begin{enumerate}
\item \label{strong convergence property}
If the filtration~\ref{filt 0 dual} is finite, 
then the spectral sequence converges strongly.
\item \label{grading property}
If $C$ is also a graded cocommutative $k$-coalgebra and the filtration layers
$F_nC$ are graded $k$-linear subspaces,
then the differential preserves the grading.
\item 
If $C$ is the underlying coalgebra of the $k$-linear dual Hopf algebra of a commutative Hopf algebra $A$ over $k$,
and the filtration \ref{filt 0 dual} is a filtration by Hopf ideals, then $E^0C$ is is also a cocommutative Hopf algebra,
and the $E_1$-term and the abutment of the spectral sequence each have a natural ring structure. The spectral sequence is furthermore one of algebras, and the product in the spectral sequence converges
to the product on the abutment. 
\end{enumerate}
\end{prop}
\begin{proof}
This is formally dual to Proposition \ref{may ss for hh}. The only thing that needs some explanation is the ring structure.
The underlying filtered DGA of this spectral sequence has a
ring structure given by the composite
\begin{align*} 
 CartC^{\bullet}(C,C) \otimes_{k} CartC^{\bullet}(C,C) &\stackrel{\cong}{\longrightarrow} \left( CC_{\bullet}(A,A)\right)^{*} \otimes_{k} \left(CC_{\bullet}(A,A)\right)^{*} \\
 & \stackrel{\cong}{\longrightarrow} \left( CC_{\bullet}(A,A) \otimes_{k} CC_{\bullet}(A,A)\right)^{*} \\
 & \stackrel{\simeq}{\longrightarrow} \left( CC_{\bullet}(A\otimes_{k} A,A\otimes_k A)\right)^{*} \\
 & \stackrel{\Delta^*}{\longrightarrow}  CC_{\bullet}(A,A)^{*} \\
 & \stackrel{\cong}{\longrightarrow} CartC^{\bullet}(C,C) ,
\end{align*}
where the map marked $\simeq$ is the $k$-linear dual of the Alexander-Whitney map, 
the map marked $\Delta^*$ is the $k$-linear dual of $CC_{\bullet}$ applied to the comultiplication map on $A$
(which is well-defined, since $CC_{\bullet}$ is functorial on $k$-algebra maps and since $A$ is assumed cocommutative, so that its
comultiplication is a $k$-algebra morphism). 
The rest is formal.
\end{proof}

\begin{prop}{\bf (May spectral sequence for dual Hochschild cohomology, with coefficients in the base field.)}\label{may ss for dual hh, supertrivial coeffs}
Let $k$ be a field and let $C$ be a $k$-coalgebra.
Suppose $C$ is equipped with a comultiplicative filtration as in \eqref{filt 0 dual}.
Then there exists a spectral sequence
\begin{align*} 
 E_1^{s,t} \cong CartH^{s,t}(E^0C, k) & \Rightarrow  CartH^s(C,k) \\
 d_r^{s,t} : E_r^{s,t} & \rightarrow E_r^{s+1,t-r}.
 \end{align*}
The bigrading superscripts
$CartH^{s,t}$ are as in Proposition~\ref{may ss for dual hh}.

This spectral sequence enjoys properties~\ref{strong convergence property} and \ref{grading property} from Proposition~\ref{may ss for dual hh}.
\end{prop}
\begin{proof}
Essentially identical to Proposition~\ref{may ss for dual hh}.
\end{proof}
Differentials in the spectral sequences of Propositions \ref{may ss for dual hh} and \ref{may ss for dual hh, supertrivial coeffs} are computable, at least in principle, by the same method described in Remark \ref{how to compute differentials}. We use that method below, in Propositions \ref{d1 diffs}, \ref{d2 differential computation}, and \ref{no more diffs}.

\section{The May and abelianizing filtrations.}

We aim to compute $HH_*(A(1), A(1))$, the Hochschild homology of $A(1)$. By Proposition~\ref{duality between hh and cohh}, this amounts to
computing $CartH^*(A(1)^*,A(1)^*)$, and then taking the $\mathbb{F}_2$-linear dual. We now go about doing this.

\begin{definition}
Recall that the {\em May filtration} on $A(1)$ is the filtration
by powers of the augmentation ideal $I = (\Sq^1,\Sq^2)$. We write ${}^{Aug}F^n(A(1))$ for the $n$th filtration layer in this filtration, i.e., 
${}^{Aug}F^nA(1) = I^n$, and we write ${}^{Aug}{E}_0A(1)$ for the
associated graded $\mathbb{F}_2$-algebra. 
If $x\in A(1)$, we sometimes write $\widetilde{x}$ for the associated element in ${}^{Aug}{E}_0A(1)$.
\end{definition}

\begin{prop}\label{assoc graded of may filt}
The $\mathbb{F}_2$-algebra ${}^{Aug}{E}_0A(1)$ is the graded $\mathbb{F}_2$-algebra with generators $\widetilde{\Sq}^1$ and $\widetilde{\Sq}^2$ in grading degrees $1$ and $2$, respectively,
and relations
\[ 0= \widetilde{\Sq}^1\widetilde{\Sq}^1= \widetilde{\Sq}^2\widetilde{\Sq}^2 = 
 \widetilde{\Sq}^1\widetilde{\Sq}^2\widetilde{\Sq}^1\widetilde{\Sq}^2 + \widetilde{\Sq}^2\widetilde{\Sq}^1\widetilde{\Sq}^2\widetilde{\Sq}^1.\]

The $\mathbb{F}_2$-linear duals of $A(1)$ and ${}^{Aug}{E}_0(A(1))$ are, as
Hopf algebras,
\begin{align*} 
A(1)^* &= \mathbb{F}_2[\overline{\xi}_1, \overline{\xi}_2]/\left(\overline{\xi}_1^4, \overline{\xi}_2^2\right), \\
\Delta(\overline{\xi}_2) &= \overline{\xi}_2\otimes 1 + \overline{\xi}_1 \otimes \overline{\xi}_1^2 + 1\otimes \overline{\xi}_2, \\
\left({}^{Aug}{E}_0A(1)\right)^* &= \mathbb{F}_2[\overline{\xi}_{1,0}, \overline{\xi}_{1,1},\overline{\xi}_{2,0}]/\left(\overline{\xi}_{1,0}^2, \overline{\xi}_{1,1}^2, \overline{\xi}_{2,0}^2\right), \\
\Delta(\overline{\xi}_{2,0}) &= \overline{\xi}_{2,0}\otimes 1 + \overline{\xi}_{1,0} \otimes \overline{\xi}_{1,1} + 1\otimes \overline{\xi}_{2,0}, \\
\end{align*}
with $\overline{\xi}_1,\overline{\xi}_{1,0},\overline{\xi}_{1,1}$
all primitive\footnote{The notation $\overline{x}$ is traditional for the conjugate of
$x$ in a Hopf algebra. In $A(1)_*$, we have $\overline{\xi}_1 = \xi_1$ 
and $\overline{\xi}_2 = \xi_2 + \xi_1^3$. Another notational point: the symbol $\overline{\xi}_{i,j}$ is used to denote the image of $\overline{\xi}_i^{2^j}\in A(1)$ in ${}^{Aug}{E}_0(A(1))$.}.
\end{prop}
\begin{proof}
Well-known consequence of the computation of the dual Steenrod algebra, as in~\cite{MR0099653}.
\end{proof}

We do not know a reference for Proposition \ref{a(1) and dihedral group}, but we doubt that it is a new observation.
\begin{prop}\label{a(1) and dihedral group}
The $\mathbb{F}_2$-algebra ${}^{Aug}{E}_0A(1)$ is isomorphic to the group ring
$\mathbb{F}_2[D_8]$ of the dihedral group $D_8$.
\end{prop}
\begin{proof}
We use the presentation
\[ D_8 = \langle a,b \mid a^2, b^2, abab = baba\rangle \]
for $D_8$.
The $\mathbb{F}_2$-algebra map
\[ f: \mathbb{F}_2[D_8] \rightarrow {}^{Aug}{E}_0A(1)\]
given by
\begin{align*} 
 f(a) &= 1 + \widetilde{\Sq}^1 \\
 f(b) &= 1 + \widetilde{\Sq}^2 \end{align*}
is well-defined\footnote{Here it is essential that we are using
${}^{Aug}{E}_0A(1)$ and not $A(1)$, since $\left(\widetilde{\Sq}^2\right)^2 = 0$
in ${}^{Aug}{E}_0A(1)$ but $(\Sq^2)^2 \neq 0$ in $A(1)$.}, since $f(a)^2 = 1 = f(b)^2$ and
$f(a)f(b)f(a)f(b) = f(b)f(a)f(b)f(a)$. 
The algebra map $f$ is surjective since it hits the generators of ${}^{Aug}{E}_0A(1)$, and its domain and codomain are finite-dimensional of the same dimension, hence $f$ must be an isomorphism.
\end{proof}

We now use the well-known computation of the Hochschild homology of group rings. See \cite{MR0814144} for this result when $k$ has characteristic zero, and Theorem 7.4.6 of \cite{MR1600246} for the general case, which is as follows:
\begin{theorem}  \label{burghelea thm}
 Suppose $G$ is a discrete group, $k$ a field. Let $\langle G \rangle$ be the set of
conjugacy classes of elements in $G$, and given a conjugacy class $S$,
let $C_G(S)$ denote the centralizer of $S$ in $G$. Then there exists an 
isomorphism of graded $k$-vector spaces
\[ HH_*(k[G],k[G]) \cong \oplus_{S\in \langle G\rangle} H_*(C_G(S); k) .\]
\end{theorem}

\begin{corollary}\label{assoc gr dimension count}
The dimension of $HH_n({}^{Aug}{E}_0A(1), {}^{Aug}{E}_0A(1))$ as a $k$-vector space is
\[ \dim_k HH_n({}^{Aug}{E}_0A(1), {}^{Aug}{E}_0A(1)) = 3n+5.\]
\end{corollary}
\begin{proof}
We use Proposition~\ref{a(1) and dihedral group} and Theorem~\ref{burghelea thm}. There are five conjugacy classes
of elements in $D_8= \langle a,b \mid a^2, b^4, ab = b^3 a\rangle$:
\[ 1, \{ a,b^2a\}, \{ ba, b^3a\}, \{ b,b^3\}, \{ b^2\},\] 
with centralizers 
\[ D_8, \langle a,b^2\mid a^2,(b^2)^2\rangle, 
  \langle b^2\mid (b^2)^2\rangle , \langle b\mid b^4\rangle , D_8,\]
respectively.
These centralizer subgroups are isomorphic to
\[ D_8, C_2\times C_2, C_2, C_4, D_8,\]
respectively. The homology of these groups with $\mathbb{F}_2$ coefficients
is well-known:
\begin{align*} 
 \dim_{\mathbb{F}_2} H_n(D_8; \mathbb{F}_2) &= n+1 \\ 
 \dim_{\mathbb{F}_2} H_n(C_2\times C_2; \mathbb{F}_2) &= n+1 \\
 \dim_{\mathbb{F}_2} H_n(C_2; \mathbb{F}_2) &= 1 \\
 \dim_{\mathbb{F}_2} H_n(C_4; \mathbb{F}_2) &= 1 ,\end{align*}
hence
\begin{align*}
 \dim_{\mathbb{F}_2} HH_n({}^{Aug}{E}_0A(1), {}^{Aug}{E}_0A(1)) 
  &= \dim_{\mathbb{F}_2} H_n(D_8; \mathbb{F}_2) + \dim_{\mathbb{F}_2} H_n(C_2\times C_2; \mathbb{F}_2) \\
  &\ \ \ \ \  + \dim_{\mathbb{F}_2} H_n(C_2; \mathbb{F}_2) + \dim_{\mathbb{F}_2} H_n(C_4; \mathbb{F}_2) \\
  &\ \ \ \ \  + \dim_{\mathbb{F}_2} H_n(D_8; \mathbb{F}_2) \\
  &= 3n+5.\end{align*}
\end{proof}

To avoid any confusion about notation in Definition \ref{def of ab filt}: $(x)$ denotes the two-sided ideal generated by an element $x$. The notation $Q_1$ is Milnor's notation for the element $\Sq^1\Sq^2+\Sq^2\Sq^1$ of $A(1)$.
\begin{definition}\label{def of ab filt}
We now define a new filtration on $A(1)$ which we will call {\em the abelianizing filtration on $A(1)$}. To notationally distinguish it 
from the May filtration, we will write ${}^{Ab}{F}^n(A(1))$ for its filtration stages,
and ${}^{Ab}{E}_0(A(1))$ for its associated graded algebra.
The abelianizing filtration is defined as follows. 
\begin{align*} 
{}^{Ab}{F}^0(A(1)) &= A(1) \\
{}^{Ab}{F}^1(A(1)) &= ( \Sq^1,\Sq^2) \\
{}^{Ab}{F}^2(A(1)) &= ( \Sq^2 ) \\
{}^{Ab}{F}^3(A(1)) &= ( \Sq^1\Sq^2, \Sq^2\Sq^1 ) \\
{}^{Ab}{F}^4(A(1)) &= ( Q_1 ) \\
{}^{Ab}{F}^5(A(1)) &= ( \Sq^1Q_1,\Sq^2Q_1) \\
{}^{Ab}{F}^6(A(1)) &= ( \Sq^2Q_1) \\
{}^{Ab}{F}^7(A(1)) &= ( \Sq^1\Sq^2Q_1) \\
{}^{Ab}{F}^8(A(1)) &= 0. \end{align*}
If $x\in A(1)$, we sometimes write $\dbtilde{x}$ for the associated element in ${}^{Ab}{E}_0A(1)$.
\end{definition}

\begin{observation}\label{assoc graded of abelianizing filt}
By routine calculation, one sees that the abelianizing filtration on $A(1)$ has the following properties:
\begin{itemize}
\item The abelianizing filtration is finer than the May filtration,
that is, ${}^{Aug}F^n(A(1))\subseteq {}^{Ab}{F}^n(A(1))$ for all $n$.
\item The abelianizing filtration is a multiplicative filtration, that
is, if $x\in {}^{Ab}{F}^m(A(1))$ and $y\in {}^{Ab}{F}^n(A(1))$, then
$xy\in {}^{Ab}{F}^{m+n}(A(1))$.
\item Furthermore, the abelianizing filtration is a Hopf filtration,
that is, if $x\in {}^{Ab}{F}^m(A(1))$, then 
\[ \Delta(x) \in \sum_{i=0}^m {}^{Ab}{F}^i(A(1)) \otimes_{\mathbb{F}_2} {}^{Ab}{F}^{m-i}(A(1)).\]
\item The associated graded Hopf algebra ${}^{Ab}{E}_0(A(1))$ of the abelianizing
filtration on $A(1)$ is the exterior algebra $E(\dbtilde{\Sq}^1, \dbtilde{\Sq}^2,\dbtilde{Q}_1)$, with $\dbtilde{\Sq}^1$ and $\dbtilde{Q}_1$ primitive, and with $\Delta(\dbtilde{\Sq}^2) = \dbtilde{\Sq}^2\otimes 1 + \dbtilde{\Sq}^1 \otimes \dbtilde{\Sq}^1 + 1\otimes \dbtilde{\Sq}^2$.
The topological degrees (i.e., inherited from the usual grading on the Steenrod algebra) of $\dbtilde{\Sq}^1, \dbtilde{\Sq}^2,\dbtilde{Q}_1$ are $1,3,$ and $4$, respectively, and their abelianizing degrees (i.e., inherited from the abelianizing filtration) 
are $1,2,$ and $4$, respectively.
\item The $\mathbb{F}_2$-linear dual Hopf algebra $\left({}^{Ab}{E}_0(A(1))\right)^*$ is 
\[ \left({}^{Ab}{E}_0(A(1))\right)^* \cong \mathbb{F}_2[\overline{\xi}_{1,0}, \overline{\xi}_{2,0}]/\left(\overline{\xi}_{1,0}^4, \overline{\xi}_{2,0}^2\right),\]
with $\overline{\xi}_{1,0},\overline{\xi}_{2,0}$ both
primitive.
In particular, $\left({}^{Ab}{E}_0(A(1))\right)^*$ is isomorphic to $\left({}^{Aug}{E}_0(A(1))\right)^*$ as $\mathbb{F}_2$-coalgebras, but not as Hopf algebras.
\end{itemize}
\end{observation}

\section{Running the HH-May and abelianizing spectral sequences.}

\subsection{Input.}

Beginning in this section, we will sometimes use the standard notations:
\begin{itemize}
\item $\Gamma_k(x_1, \dots ,x_n)$ denotes the divided power $k$-algebra on generators $x_1, \dots ,x_n$,
\item $E(x_1, \dots ,x_n)$ denotes the exterior $k$-algebra on generators $x_1, \dots ,x_n$,
\item and $P(x_1, \dots ,x_n)$ denotes the polynomial $k$-algebra on generators $x_1, \dots ,x_n$.
\end{itemize}
\begin{lemma}\label{hochschild homology of exterior alg}
Let $k$ be a field of characteristic two, and give $k[x]/x^2$ the structure of a Hopf algebra over $k$ by letting $\Delta(x) = x\otimes 1 + 1\otimes x$.
Then we have an isomorphism of graded Hopf algebras
\[ HH_*(k[x]/x^2, k[x]/x^2) \cong k[x]/x^2\otimes_k \Gamma_k(\sigma x),\]
where:
\begin{itemize}
\item $HH_*(k[x]/x^2, k[x]/x^2)$ inherits its coproduct from that of \linebreak $k[x]/x^2$, as in Remark~\ref{explicit coproduct remark},
\item $\sigma x$ denotes the homology class of the $1$-cycle $1[x]$, i.e., $1\otimes x$ (and consequently $\sigma x$ is in homological degree $1$),
\item $\Delta(x) = x\otimes 1 + 1\otimes x$, and 
\item $\Delta(\gamma_n(\sigma x)) = \sum_{i=0}^n \gamma_i(\sigma x) \otimes \gamma_{n-i}(\sigma x)$, where $\gamma_n$ is the $n$th divided power.
\end{itemize}
As a consequence, the Hopf $k$-algebra $CartH^*\left((k[x]/x^2)^*,(k[x]/x^2)^*\right)$ is isomorphic to $k[\xi]/\xi^2\otimes_k k[h]$, with $\xi,h$ both primitive. Here $\xi$ denotes the dual basis element to $x$, and $h$ is the dual basis element to $\sigma x$. Consequently $\xi,h$ are in cohomological degrees $0,1$, respectively.
\end{lemma}
\begin{proof}
Using the cycle representative $1\left[ x\otimes \dots \otimes x\right]$ for $\gamma_n(\sigma x)$, with $n$ tensor factors of $x$, the claims made are
easy (and classical) computations of coproducts and shuffle products in the Hochschild chain complex, using Remark~\ref{explicit coproduct remark}.
\end{proof}

\begin{lemma}\label{hochschild homology of truncated poly alg}
Let $k$ be a field of characteristic two, and give $k[x,y]/(x^2,y^2)$ the structure of a Hopf algebra over $k$ by letting $\Delta(x) = x\otimes 1 + 1\otimes x$
and by letting $\Delta(y) = y\otimes 1 + x\otimes x + 1\otimes y$.
Then we have an isomorphism of graded Hopf algebras
\[ HH_*\left(k[x,y]/(x^2,y^2), k[x,y]/(x^2,y^2)\right) \cong k[x,y]/(x^2,y^2) \otimes_k \Gamma_k(\sigma x,\sigma y),\]
where:
\begin{itemize}
\item $\sigma x$ (respectively, $\sigma y$) denotes the homology class of the $1$-cycle $1[ x]$ (respectively, $1[y]$), and consequently $\sigma x$ and $\sigma y$ are each in homological degree $1$,
\item the coproducts $\Delta(x)$ and $\Delta(y)$ are exactly as in $k[x,y]/(x^2,y^2)$, 
\item $\Delta(\gamma_n(\sigma x)) = \sum_{i=0}^n \gamma_i(\sigma x) \otimes \gamma_{n-i}(\sigma x)$, and
\item $\Delta(\gamma_n(\sigma y)) = \sum_{i=0}^n \gamma_i(\sigma y) \otimes \gamma_{n-i}(\sigma y)$.
\end{itemize}
Consequently, the Hopf $k$-algebra $CartH^*\left( (k[x,y]/(x^2,y^2))^*, (k[x,y]/(x^2,y^2))^*\right)$ is isomorphic to $k[\xi]/\xi^4\otimes_k k[h_x,h_y]$, with $\xi,h_x,h_y$ all primitive.
Here $\xi$ is the dual basis element to $x$, and $h_x,h_y$ are the dual basis elements to $\sigma x,\sigma y$, respectively. Consequently $\xi,h_x,h_y$ are in grading degrees $0,1,1$, respectively.
\end{lemma}
\begin{proof}
Just as in Lemma~\ref{hochschild homology of exterior alg}.
(Note that $\Delta(\gamma_n(\sigma y))$ does not involve $x$ or $\sigma x$, even though $\Delta(y)$ involves $x$; this is simply because
the formula for $\Delta(\sigma y) = \Delta\left( 1\left[ y\right]\right)$, given in Remark~\ref{explicit coproduct remark}, does not actually make any use of the coproduct of $y$ in $k[x,y]/(x^2,y^2)$.)
\end{proof}

In the spectral sequences described in Proposition \ref{SSs and cocycle reps}, the tridegrees $(s,t,u)$ are as follows: $s$ is the cohomological degree, $t$ is the filtration degree, and $u$ is the internal/topological degree arising from the grading on the algebra $A(1)$ itself.
\begin{prop}\label{SSs and cocycle reps}
There exist four strongly convergent trigraded multiplicative spectral sequences, each with differential of the form $d_r^{s,t,u}  :  E_r^{s,t,u} \rightarrow E_r^{s+1,t-r, u}$:

\hspace*{-1.23cm}
\begin{tabular}{|c|c|c|}
\hline
 Name & $E_1^{s,t,u}$ & Abutment \\
\hline
\hline
 Abelianizing & $CartH^{s,t,u}\left({}^{Ab}{E}^0(A(1)^*),{}^{Ab}{E}^0(A(1)^*)\right)$ & $CartH^s(A(1)^*,A(1)^*)$\\
 &  $\cong \left(\mathbb{F}_2[x_{10}]/(x_{10}^4) \otimes_{\mathbb{F}_2} E(x_{20})\right.$ & \\
 & $\ \ \ \ \ \ \left.\otimes_{\mathbb{F}_2} P(h_{10},h_{11},h_{20})\right)^{s,t,u}$ & \\
\hline
HH-May  & $CartH^{s,t,u}\left({}^{Aug}{E}^0(A(1)^*),{}^{Aug}{E}^0(A(1)^*)\right)$ & $CartH^s(A(1)^*,A(1)^*)$ \\
\hline
HH-May with & $CartH^{s,t,u}\left({}^{Aug}{E}^0(A(1)^*),\mathbb{F}_2)\right)$ & $CartH^s(A(1)^*,\mathbb{F}_2)$ \\
 coeffs. in $\mathbb{F}_2$ &&\\
\hline
Abelianizing 
 & $CartH^{s,t,u}\left({}^{Ab}{E}^0(A(1)^*),{}^{Ab}{E}^0(A(1)^*)\right)$ & $CartH^s({}^{Aug}{E}^0A(1)^*,{}^{Aug}{E}^0A(1)^*)$ \\ 
 to HH-May & $\cong\left(\mathbb{F}_2[x_{10}]/(x_{10}^4) \otimes_{\mathbb{F}_2} E(x_{20})\right.$ & \\ & $\ \ \ \ \ \ \left. \otimes_{\mathbb{F}_2} P(h_{10},h_{11},h_{20})\right)^{s,t,u}$ &\\
\hline
\end{tabular}

Furthermore, there exists a morphism of spectral sequences from the HH-May spectral sequence to the HH-May spectral sequence with coefficients in $\mathbb{F}_2$.
\end{prop}
\begin{proof}
Consequence of Propositions \ref{may ss for dual hh} and \ref{assoc graded of may filt}, Observation \ref{assoc graded of abelianizing filt}, and Lemmas \ref{hochschild homology of exterior alg} and~\ref{hochschild homology of truncated poly alg}.
\end{proof}
Of course the machinery we are using produces much more than just the four spectral sequences listed in Proposition \ref{SSs and cocycle reps}, but those four are the ones we will actually use in this paper.

In Proposition \ref{SSs and cocycle reps} and elsewhere, we write $x_{i0}$ for the cohomology class of the Cartier $0$-cocycle $\overline{\xi}_{i,0}$, and $h_{ij}$ for the cohomology class of the Cartier $1$-cocycle $1\left[ \overline{\xi}_{i,j}\right]$, where $\overline{\xi}_{1,0},\overline{\xi}_{1,1},\overline{\xi}_{2,0}$ are the dual basis elements to the elements $\dbtilde{\Sq}^1,\dbtilde{\Sq}^2,\dbtilde{Q}_1$, respectively, in ${}^{Ab}{E}_0A(1)$.

The tridegrees of the generators of $CartH^{*,*,*}\left({}^{Ab}{E}^0(A(1)^*),{}^{Ab}{E}^0(A(1)^*)\right)$, and cocycle representatives for those cohomology classes in the
Cartier cochain complex, are as follows:
\begin{align*} 
x_{10} = [\overline{\xi}_{1,0}] & \in CartC^{0,1,1}\left({}^{Ab}{E}^0(A(1)^*),{}^{Ab}{E}^0(A(1)^*)\right) \\
x_{20} = [\overline{\xi}_{2,0}] & \in CartC^{0,4,3}\left({}^{Ab}{E}^0(A(1)^*),{}^{Ab}{E}^0(A(1)^*)\right) \\
h_{10} = [1\otimes \overline{\xi}_{1,0}] & \in CartC^{1,1,1}\left({}^{Ab}{E}^0(A(1)^*),{}^{Ab}{E}^0(A(1)^*)\right) \\
h_{11} = [1\otimes \overline{\xi}_{1,1}] & \in CartC^{1,2,2}\left({}^{Ab}{E}^0(A(1)^*),{}^{Ab}{E}^0(A(1)^*)\right) \\
h_{20} = [1\otimes \overline{\xi}_{2,0}] & \in CartC^{1,4,3}\left({}^{Ab}{E}^0(A(1)^*),{}^{Ab}{E}^0(A(1)^*)\right) \end{align*}

\begin{prop}\label{ss comparison with supertrivial coeffs}
The HH-May spectral sequence with coefficients in $\mathbb{F}_2$, from Proposition~\ref{SSs and cocycle reps}, is isomorphic (beginning with the $E_1$-term) to the 
classical May spectral sequence\footnote{A good reference for this classical May spectral sequence is Example 3.2.7 of \cite{MR860042}.} for $A(1)$, $\Ext_{E^0A(1)}^{*,*,*}(\mathbb{F}_2,\mathbb{F}_2) \Rightarrow \Ext_{A(1)}^{*,*}(\mathbb{F}_2,\mathbb{F}_2)$. 
\end{prop}
\begin{proof}
By Proposition~\ref{duality between hh and cohh}, the HH-May spectral sequence with coefficients in $\mathbb{F}_2$ has
input
\begin{align*}  E_1^{s,t,u} 
 &\cong CartH^{s,t,u}\left({}^{Aug}{E}^0(A(1)^*),\mathbb{F}_2\right) \\
 &\cong \hom_{\mathbb{F}_2}\left( HH_{s,t,u}\left({}^{Aug}{E}_0A(1),\mathbb{F}_2\right),\mathbb{F}_2\right) \\
 &\cong \hom_{\mathbb{F}_2}\left( \Tor_{s,t,u}^{{}^{Aug}{E}_0A(1)\otimes_{\mathbb{F}_2} {}^{Aug}{E}_0A(1)^{\op}}\left({}^{Aug}{E}_0A(1),\mathbb{F}_2\right),\mathbb{F}_2\right) \\
 &\cong \hom_{\mathbb{F}_2}\left( \Tor_{s,t,u}^{{}^{Aug}{E}_0A(1)^{\op}}\left(\mathbb{F}_2,\mathbb{F}_2\right),\mathbb{F}_2\right) \\
 &\cong \Ext^{s,t,u}_{{}^{Aug}{E}_0A(1)}\left(\mathbb{F}_2,\mathbb{F}_2\right),\end{align*}
using the usual $\Ext$-$\Tor$ duality properties of finite-dimensional Hopf algebras (in this case, ${}^{Aug}{E}_0A(1)$).
The same analysis on the abutment of the spectral sequence yields
\begin{align*}  E_1^{s,t,u} 
 &\cong CartH^{s,t,u}\left(A(1)^*,\mathbb{F}_2\right) \\
 &\cong \Ext^{s,t,u}_{A(1)}\left(\mathbb{F}_2,\mathbb{F}_2\right),\end{align*}
so the $E_1$-term of the HH-May spectral sequence is isomorphic to the $E_1$-term of the classical May spectral sequence for $A(1)$, and their
abutments also are isomorphic.
The fact that the spectral sequences themselves are isomorphic is due to the easy observation that the Cartier cochain complex of $A(1)$ with coefficients in $\mathbb{F}_2$ is isomorphic to the classical cobar complex of $A(1)$, as in Definition~A.1.2.11 of~\cite{MR860042}, and the May filtration on one coincides with the May filtration on the other.
\end{proof}

\subsection{$d_1$-differentials.}

\begin{prop}\label{d1 diffs}
In both the abelianizing spectral sequence and the abelianizing-to-HH-May spectral sequence, 
the $d_1$ differentials are given on the multiplicative generators by
\begin{align*}
 d_1(x_{10}) &= 0, \\
 d_1(x_{20}) &= x_{10}h_{11} + x_{10}^2h_{10}, \\
 d_1(h_{10}) &= 0, \\
 d_1(h_{11}) &= 0,\mbox{\ and} \\
 d_1(h_{20}) &= h_{10}h_{11}.\end{align*}
Using these formulas and the Leibniz rule, we get the $d_1$ differential on all elements
of the $E_1$-terms of the abelianizing and abelianizing-to-HH-May spectral sequences.
\end{prop}
\begin{proof}
In Proposition~\ref{SSs and cocycle reps} we gave cocycle representatives for the six multiplicative generators.
We then easily compute the $d_1$ differentials on those generators using the method described in Proposition~\ref{may ss for dual hh}:
\begin{align*}
 d(\overline{\xi}_{1,0}) &= 0 ,\\
 d(\overline{\xi}_{2,0}) &= \overline{\xi}_{1,0}\otimes \overline{\xi}_{1,1} + \overline{\xi}_{1,1}\otimes \overline{\xi}_{1,0} \\
                        &    + 1\otimes \overline{\xi}_{2,0} + \overline{\xi}_{2,0}\otimes 1 \\
                        &    + \overline{\xi}_{2,0}\otimes 1 + 1\otimes \overline{\xi}_{2,0} \\
                        &= \overline{\xi}_{1,0}\otimes \overline{\xi}_{1,1} + \overline{\xi}_{1,1}\otimes \overline{\xi}_{1,0} \\
 d(1\otimes \overline{\xi}_{1,0}) &= 0 ,\\
 d(1\otimes \overline{\xi}_{1,1}) &= 0 ,\\
 d(1\otimes \overline{\xi}_{2,0}) &= 1\otimes 1\otimes \overline{\xi}_{2,0} + 1\otimes \overline{\xi}_{2,0} \otimes 1 \\
                        &    + 1\otimes \overline{\xi}_{1,0}\otimes \overline{\xi}_{1,1} + 1\otimes 1\otimes \overline{\xi}_{2,0} \\
                        &    + 1\otimes \overline{\xi}_{2,0}\otimes 1 \\
                        &=     1\otimes \overline{\xi}_{1,0}\otimes \overline{\xi}_{1,1}. \end{align*}
Here $\overline{\xi}_{1,1}$ is the dual basis element to $\Sq^2$, so $\overline{\xi}_{1,1} = \overline{\xi}_{1,0}^2$ in the associated graded of the abelianizing filtration, but $\overline{\xi}_{1,1}$ is indecomposable in the associated graded of the May filtration.
Using the product on dual Hochschild cohomology from Proposition~\ref{may ss for dual hh}, we get that these cocycles
represent the cohomology classes $0, 0, x_{10}h_{11} + x_{10}^2h_{10}, 0, h_{10}h_{11}$, respectively.
\end{proof}

Now one has enough information to do a routine computation of the cohomology of the $E_1$-term, and get the $E_2$-term.
While $h_{20}$ is not a cocycle in the $E_1$-term, its square is, and we follow the traditional (due to May's thesis) 
notational conventions of May spectral sequences by writing $b_{20}$ for $h_{20}^2$.

We will present the $E_2$-term as a spectral sequence chart.
\begin{convention}\label{conventions on ss charts}
In all the spectral sequence charts in this paper,
\begin{itemize}
\item the vertical axis is the homological degree $s$,
\item the horizontal axis is the Adams degree $u-s$, i.e., the internal/topological degree $u$ minus the homological degree $s$,
\item horizontal lines (whether curved or straight) represent multiplication by $x_{10}$, 
\item vertical lines represent multiplication by $h_{10}$, and
\item diagonal lines represent multiplication by $h_{11}$.
\end{itemize}
\end{convention}

Here is a spectral sequence chart illustrating the $E_2$-term of the abelianizing and abelianizing-to-HH-May spectral sequences (their $E_2$-terms are abstractly isomorphic as bigraded $\mathbb{F}_2$-vector spaces, if one forgets about the filtration degree and only keeps track of the cohomological and Adams degrees), 
reduced modulo the ideal generated by $b_{20}$:

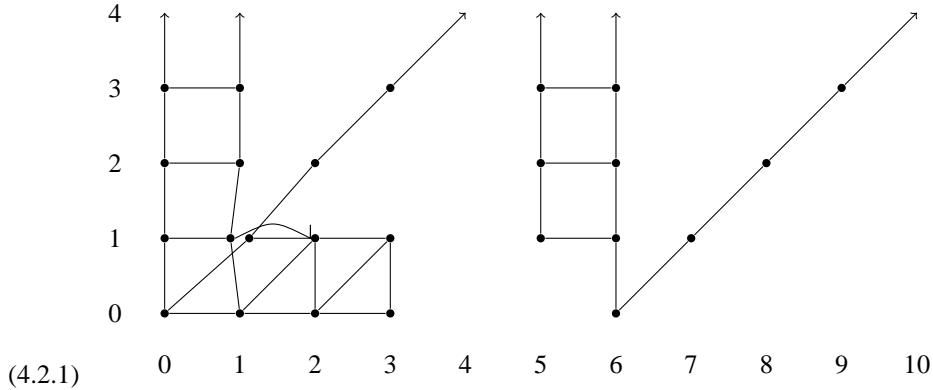
\begin{equation}\begin{sseq}[grid=none,entrysize=10mm,labelstep=1]{0...10}{0...4} \label{ss e2}
 \ssmoveto 0 0 
 \ssdropbull
 \ssname{1}

 \ssmove 1 0
 \ssdropbull 
 \ssname{x10} 

 \ssmove 1 0
 \ssdropbull
 \ssname{x11}

 \ssmove 1 0
 \ssdropbull
 \ssname{x10x11}

 \ssmove 3 0
 \ssdropbull
 \ssname{x6}

 \ssmove{-6}{1} 
 \ssdropbull
 \ssname{h10}

 \ssmove 1 0
 \ssdropbull 
 \ssname{x10h10} 
 \ssdropbull
 \ssname{h11}

 \ssmove 1 0
 \ssdropbull
 \ssname{x11h10}

 \ssmove 1 0
 \ssdropbull
 \ssname{x11h11}

 \ssmove 2 0
 \ssdropbull
 \ssname{z}

 \ssmove 1 0
 \ssdropbull
 \ssname{h10x6}

 \ssmove 1 0
 \ssdropbull
 \ssname{h11x6}

 \ssmove{-7}{1}
 \ssdropbull
 \ssname{h10^2}
 
 \ssmove 1 0
 \ssdropbull
 \ssname{x10h10^2}

 \ssmove 1 0
 \ssdropbull
 \ssname{h11^2}

 \ssmove 3 0
 \ssdropbull
 \ssname{h10z}

 \ssmove 1 0
 \ssdropbull
 \ssname{h10^2x6}

 \ssmove 2 0
 \ssdropbull
 \ssname{h11^2x6}


 \ssmove{-8}{1}
 \ssdropbull
 \ssname{h10^3}

 \ssmove 1 0
 \ssdropbull
 \ssname{x10h10^3}

 \ssmove 2 0
 \ssdropbull
 \ssname{h11^3}

 \ssmove 2 0
 \ssdropbull
 \ssname{h10^2z}

 \ssmove 1 0
 \ssdropbull
 \ssname{h10^3x6}

 \ssmove 3 0
 \ssdropbull
 \ssname{h11^3x6}
 
 \ssgoto{1}
 \ssgoto{x10}
 \ssstroke
 \ssgoto{x10}
 \ssgoto{x11}
 \ssstroke
 \ssgoto{x11}
 \ssgoto{x10x11}
 \ssstroke
 \ssgoto{1}
 \ssgoto{h10}
 \ssstroke
 \ssgoto{1}
 \ssgoto{h11}
 \ssstroke

 \ssgoto{x10}
 \ssgoto{x10h10}
 \ssstroke
 \ssgoto{x10}
 \ssgoto{x11h10}
 \ssstroke

 \ssgoto{x11}
 \ssgoto{x11h10}
 \ssstroke
 \ssgoto{x11}
 \ssgoto{x11h11}
 \ssstroke
 \ssgoto{x10x11}
 \ssgoto{x11h11}
 \ssstroke

 \ssgoto{x6}
 \ssgoto{h10x6}
 \ssstroke
 \ssgoto{x6}
 \ssgoto{h11x6}
 \ssstroke
 \ssgoto{h11x6}
 \ssgoto{h11^2x6}
 \ssstroke
 \ssgoto{h11^2x6}
 \ssgoto{h11^3x6}
 \ssstroke
 \ssgoto{h11^3x6}
 \ssmove 1 1
 \ssstroke[void,arrowto]

 \ssgoto{h10}
 \ssgoto{h10^2}
 \ssstroke
 \ssgoto{h10}
 \ssgoto{x10h10}
 \ssstroke
 \ssgoto{x10h10}
 \ssgoto{x11h10}
 \ssstroke[curve=.1]

 \ssgoto{x10h10}
 \ssgoto{x10h10^2}
 \ssstroke

 \ssgoto{h11}
 \ssgoto{x11h10}
 \ssstroke
 \ssgoto{x11h10}
 \ssgoto{x11h11}
 \ssstroke
 \ssgoto{h11}
 \ssgoto{h11^2}
 \ssstroke

 \ssgoto{h10^2}
 \ssgoto{h10^3}
 \ssstroke
 \ssgoto{h10^2}
 \ssgoto{x10h10^2}
 \ssstroke

 \ssgoto{x10h10^2}
 \ssgoto{x10h10^3}
 \ssstroke

 \ssgoto{h11^2}
 \ssgoto{h11^3}
 \ssstroke

 \ssgoto{h10^3}
 \ssgoto{x10h10^3}
 \ssstroke

 \ssgoto{z}
 \ssgoto{h10x6}
 \ssstroke
 \ssgoto{h10z}
 \ssgoto{h10^2x6}
 \ssstroke
 \ssgoto{h10^2z}
 \ssgoto{h10^3x6}
 \ssstroke

 \ssgoto{z}
 \ssgoto{h10z}
 \ssstroke
 \ssgoto{h10z}
 \ssgoto{h10^2z}
 \ssstroke
 \ssgoto{h10^2z}
 \ssmove 0 1
 \ssstroke[void,arrowto]

 \ssgoto{h10x6}
 \ssgoto{h10^2x6}
 \ssstroke
 \ssgoto{h10^2x6}
 \ssgoto{h10^3x6}
 \ssstroke
 \ssgoto{h10^3x6}
 \ssmove 0 1
 \ssstroke[void,arrowto]

 \ssgoto{h10^3}
 \ssmove 0 1
 \ssstroke[void,arrowto]

 \ssgoto{x10h10^3}
 \ssmove 0 1
 \ssstroke[void,arrowto]

 \ssgoto{h11^3}
 \ssmove 1 1
 \ssstroke[void,arrowto]

\end{sseq}\end{equation}

The classes whose names are not implied by the lines representing various multiplications are as follows:
\begin{itemize}
\item the class in bidegree $(s=1,u-s=5)$ is 
$x_{10}^2(h_{10}x_{20} + x_{10}h_{20})$, which we abbreviate as $z$,
\item and the class in bidegree $(s=0,u-s=6)$ is $x_{10}^3x_{20}$, which we abbreviate as $x_6$.
\end{itemize}

The spectral sequence's $E_2$-term is $b_{20}$-periodic, that is, there exists a class (not pictured) $b_{20}$ in bidegree $(s=2,u-s=4)$ each of whose positive integer
powers generates an isomorphic copy of the chart~\ref{ss e2}.

Consequently, as a trigraded $\mathbb{F}_2$-algebra, the spectral sequence's $E_2$-term 
is isomorphic to:
\begin{align*} \mathbb{F}_2[x_{10},h_{10},h_{11},z,x_6,b_{20}] &\mbox{\ modulo relations\ } x_{10}^4,x_{10}h_{11}=x_{10}^2h_{10}, \\
 & h_{10}h_{11},x_{10}z = h_{10}x_6, h_{11}z,z^2, \\
 & x_{10}x_6, zx_6,x_6^2,\end{align*}
with generators in tridegrees:

\begin{tabular}{c|c|c|c|c}
\mbox{Class} & \mbox{Coh. degree $s$} & \mbox{Ab. degree $t$} & \mbox{Top. degree $u$} & \mbox{Adams degree $u-s$} \\
$x_{10}$ & 0 & 1 & 1 & 1 \\
  $x_6$ & 0 & 7 & 6 & 6 \\
$h_{10}$ & 1 & 1 & 1 & 0 \\
$h_{11}$ & 1 & 2 & 2 & 1 \\
    $z$ & 1 & 7 & 6 & 5 \\
$b_{20}$ & 2 & 8 & 6 & 4  
\end{tabular}

\subsection{$d_2$-differentials.}

\begin{prop}
The abelianizing-to-HH-May spectral sequence collapses at $E_2$, i.e., there are no nonzero differentials longer than $d_1$ differentials.
Consequently, the spectral sequence chart~\ref{ss e2} describes the $E_1$-term (and also the $E_2$-term) of the HH-May spectral sequence, as well as
the $E_2$-term of the abelianizing spectral sequence.
\end{prop}
\begin{proof}
An easy dimension count on the $E_2$-term~\ref{ss e2} gives us that the $\mathbb{F}_2$-vector space dimension of the $s$-row is
$3s+5$. By Proposition~\ref{duality between hh and cohh} and 
Corollary~\ref{assoc gr dimension count}, this is the correct dimension for the $E_{\infty}$-term. So there can be no further nonzero differentials
in the spectral sequence, since any such differentials would reduce the $\mathbb{F}_2$-vector space dimension of some row.
\end{proof}

\begin{prop}\label{d2 differential computation}
The $d_2$ differentials on the multiplicative generators of the $E_2$-term of the HH-May spectral sequence, as well as the abelianizing spectral sequence, are as follows:
\begin{align*}
 d_2(x_{10}) &= 0, \\
 d_2(h_{10}) &= 0, \\
 d_2(h_{11}) &= 0, \\
 d_2(z) &= 0, \\
 d_2(x_{6}) &= 0,\mbox{\ and} \\
 d_2(b_{20}) &= h_{11}^3. \end{align*}
Using these formulas and the Leibniz rule, we get the $d_2$ differential on all elements
of the $E_2$-term of the abelianizing spectral sequence.
\end{prop}
\begin{proof}
For $x_{10},h_{10},$ and $h_{11}$, this is simply the same computation as Proposition~\ref{d1 diffs}. The point is that,
\begin{itemize}
\item when we take a cocycle representative for any of these three cohomology classes in the Cartier cochain complex for ${}^{Aug}{E}_0A(1)$, 
\item then regard that cocycle as a cochain in the Cartier cochain complex for $A(1)$, 
\item the resulting cochain is still a cocycle.
\end{itemize}
Consequently the differential $d_r$ vanishes, for all $r$, on each of the three cohomology classes $x_{10},h_{10},$ and $h_{11}$.

For $z$ and $x_6$, inspection of the tridegrees of elements rules out all nonzero possibilities for $d_2$. 
For the differential $d_2(b_{20})$: we see from inspection of the tridegrees that the only possible nonzero differential on $b_{20}$ would
have to hit a scalar multiple of $h_{11}^3$, and this differential indeed occurs, using Proposition~\ref{ss comparison with supertrivial coeffs} to 
map the HH-May spectral sequence to the classical May spectral sequence for $A(1)$, in which the differential $d_2(b_{20}) = h_{11}^3$ is
classical and well-known (see e.g. Lemma~3.2.10 of~\cite{MR860042}).
\end{proof}

So the only nonzero $d_2$ differentials are  the $d_2$-differential $d_2(b_{20}) = h_{11}^3$ and its products
with other classes. By the Leibniz rule, $d_2(b_{20}^2) = 0$, so the spectral sequence's $E_3$-term is $b_{20}^2$-periodic. We now draw a chart illustrating
the $E_3$-term, modulo the two-sided ideal generated by $b_{20}^2$:

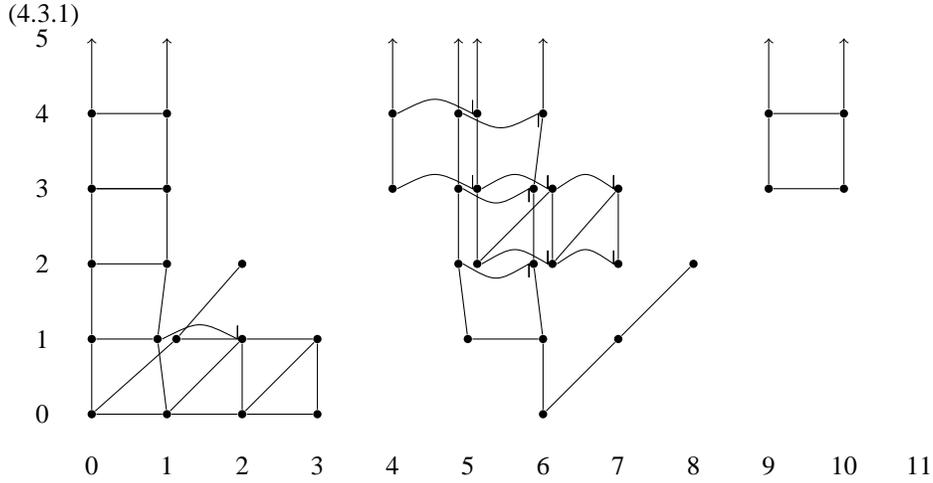
\begin{equation}\begin{sseq}[grid=none,entrysize=10mm,labelstep=1]{0...11}{0...5}\label{ss e3}
 \ssmoveto 0 0 
 \ssdropbull
 \ssname{1}

 \ssmove 1 0
 \ssdropbull 
 \ssname{x10} 

 \ssmove 1 0
 \ssdropbull
 \ssname{x11}

 \ssmove 1 0
 \ssdropbull
 \ssname{x10x11}

 \ssmove 3 0
 \ssdropbull
 \ssname{x6}

 \ssmove{-6}{1} 
 \ssdropbull
 \ssname{h10}

 \ssmove 1 0
 \ssdropbull 
 \ssname{x10h10} 
 \ssdropbull
 \ssname{h11}

 \ssmove 1 0
 \ssdropbull
 \ssname{x11h10}

 \ssmove 1 0
 \ssdropbull
 \ssname{x11h11}

 \ssmove 2 0
 \ssdropbull
 \ssname{z}

 \ssmove 1 0
 \ssdropbull
 \ssname{h10x6}

 \ssmove 1 0
 \ssdropbull
 \ssname{h11x6}

 \ssmove{-7}{1}
 \ssdropbull
 \ssname{h10^2}
 
 \ssmove 1 0
 \ssdropbull
 \ssname{x10h10^2}

 \ssmove 1 0
 \ssdropbull
 \ssname{h11^2}

 \ssmove 3 0
 \ssdropbull
 \ssname{h10z}
 \ssdropbull
 \ssname{x10b20}

 \ssmove 1 0
 \ssdropbull
 \ssname{h10^2x6}
 \ssdropbull
 \ssname{x10^2b20}

 \ssmove 1 0
 \ssdropbull
 \ssname{x10^3b20}

 \ssmove 1 0
 \ssdropbull
 \ssname{h11^2x6}


 \ssmove{-8}{1}
 \ssdropbull
 \ssname{h10^3}

 \ssmove 1 0
 \ssdropbull
 \ssname{x10h10^3}

 \ssmove 3 0
 \ssdropbull
 \ssname{h10b20}

 \ssmove 1 0 
 \ssdropbull
 \ssname{h10^2z}
 \ssdropbull
 \ssname{x10h10b20}

 \ssmove 1 0
 \ssdropbull
 \ssname{h10^3x6}
 \ssdropbull
 \ssname{x10^2h10b20}

 \ssmove 1 0
 \ssdropbull
 \ssname{x10^3h10b20}

 \ssmove 2 0
 \ssdropbull
 \ssname{h10b20z}

 \ssmove 1 0
 \ssdropbull
 \ssname{x10h10b20z}

 \ssmove{-10}{1}
 \ssdropbull
 \ssname{h10^4}

 \ssmove 1 0
 \ssdropbull
 \ssname{x10h10^4}

 \ssmove 3 0
 \ssdropbull
 \ssname{h10^2b20}

 \ssmove 1 0
 \ssdropbull
 \ssname{h10^3z}
 \ssdropbull
 \ssname{x10h10^2b20}

 \ssmove 1 0
 \ssdropbull
 \ssname{h10^4x6}

 \ssmove 3 0
 \ssdropbull
 \ssname{h10^2b20z}

 \ssmove 1 0
 \ssdropbull
 \ssname{x10h10^2b20z}

 \ssgoto{1}
 \ssgoto{x10}
 \ssstroke
 \ssgoto{x10}
 \ssgoto{x11}
 \ssstroke
 \ssgoto{x11}
 \ssgoto{x10x11}
 \ssstroke
 \ssgoto{1}
 \ssgoto{h10}
 \ssstroke
 \ssgoto{1}
 \ssgoto{h11}
 \ssstroke

 \ssgoto{x10}
 \ssgoto{x10h10}
 \ssstroke
 \ssgoto{x10}
 \ssgoto{x11h10}
 \ssstroke

 \ssgoto{x11}
 \ssgoto{x11h10}
 \ssstroke
 \ssgoto{x11}
 \ssgoto{x11h11}
 \ssstroke
 \ssgoto{x10x11}
 \ssgoto{x11h11}
 \ssstroke

 \ssgoto{x6}
 \ssgoto{h10x6}
 \ssstroke
 \ssgoto{x6}
 \ssgoto{h11x6}
 \ssstroke
 \ssgoto{h11x6}
 \ssgoto{h11^2x6}
 \ssstroke

 \ssgoto{h10}
 \ssgoto{h10^2}
 \ssstroke
 \ssgoto{h10}
 \ssgoto{x10h10}
 \ssstroke
 \ssgoto{x10h10}
 \ssgoto{x11h10}
 \ssstroke[curve=.1]

 \ssgoto{x10h10}
 \ssgoto{x10h10^2}
 \ssstroke

 \ssgoto{h11}
 \ssgoto{x11h10}
 \ssstroke
 \ssgoto{x11h10}
 \ssgoto{x11h11}
 \ssstroke
 \ssgoto{h11}
 \ssgoto{h11^2}
 \ssstroke

 \ssgoto{h10^2}
 \ssgoto{h10^3}
 \ssstroke
 \ssgoto{h10^3}
 \ssgoto{h10^4}
 \ssstroke
 \ssgoto{h10^2}
 \ssgoto{x10h10^2}
 \ssstroke
 \ssgoto{h10^3}
 \ssgoto{x10h10^3}
 \ssstroke

 \ssgoto{x10h10^2}
 \ssgoto{x10h10^3}
 \ssstroke
 \ssgoto{x10h10^3}
 \ssgoto{x10h10^4}
 \ssstroke

 \ssgoto{h10^3}
 \ssgoto{x10h10^3}
 \ssstroke
 \ssgoto{h10^4}
 \ssgoto{x10h10^4}
 \ssstroke

 \ssgoto{z}
 \ssgoto{h10x6}
 \ssstroke
 \ssgoto{h10z}
 \ssgoto{h10^2x6}
 \ssstroke[curve=-.1]
 \ssgoto{h10^2z}
 \ssgoto{h10^3x6}
 \ssstroke[curve=-.1]
 \ssgoto{h10^3z}
 \ssgoto{h10^4x6}
 \ssstroke[curve=-.1]

 \ssgoto{z}
 \ssgoto{h10z}
 \ssstroke
 \ssgoto{h10z}
 \ssgoto{h10^2z}
 \ssstroke
 \ssgoto{h10^2z}
 \ssgoto{h10^3z}
 \ssstroke

 \ssgoto{h10x6}
 \ssgoto{h10^2x6}
 \ssstroke
 \ssgoto{h10^2x6}
 \ssgoto{h10^3x6}
 \ssstroke
 \ssgoto{h10^3x6}
 \ssgoto{h10^4x6}
 \ssstroke

 \ssgoto{h10b20}
 \ssgoto{x10h10b20}
 \ssstroke[curve=.1]
 \ssgoto{x10h10b20}
 \ssgoto{x10^2h10b20}
 \ssstroke[curve=.1]
 \ssgoto{x10^2h10b20}
 \ssgoto{x10^3h10b20}
 \ssstroke
 \ssgoto{x10b20}
 \ssgoto{x10^2b20}
 \ssstroke[curve=.1]
 \ssgoto{x10^2b20}
 \ssgoto{x10^3b20}
 \ssstroke
 \ssgoto{h10^2b20}
 \ssgoto{x10h10^2b20}
 \ssstroke[curve=.1]
 \ssgoto{h10b20z}
 \ssgoto{x10h10b20z}
 \ssstroke
 \ssgoto{h10^2b20z}
 \ssgoto{x10h10^2b20z}
 \ssstroke

 \ssgoto{h10b20}
 \ssgoto{h10^2b20}
 \ssstroke
 \ssgoto{x10h10b20}
 \ssgoto{x10h10^2b20}
 \ssstroke
 \ssgoto{x10b20}
 \ssgoto{x10h10b20}
 \ssstroke
 \ssgoto{x10^2b20}
 \ssgoto{x10^2h10b20}
 \ssstroke
 \ssgoto{x10^3b20}
 \ssgoto{x10^3h10b20}
 \ssstroke
 \ssgoto{h10b20z}
 \ssgoto{h10^2b20z}
 \ssstroke
 \ssgoto{x10h10b20z}
 \ssgoto{x10h10^2b20z}
 \ssstroke

 \ssgoto{x10b20}
 \ssgoto{x10^2h10b20}
 \ssstroke
 \ssgoto{x10^2b20}
 \ssgoto{x10^3h10b20}
 \ssstroke

 \ssgoto{h10^4}
 \ssmove 0 1
 \ssstroke[void,arrowto]
 \ssgoto{x10h10^4}
 \ssmove 0 1
 \ssstroke[void,arrowto]

 \ssgoto{h10^3z}
 \ssmove 0 1
 \ssstroke[void,arrowto]
 \ssgoto{h10^4x6}
 \ssmove 0 1
 \ssstroke[void,arrowto]
 \ssgoto{h10^2b20}
 \ssmove 0 1
 \ssstroke[void,arrowto]
 \ssgoto{x10h10^2b20}
 \ssmove 0 1
 \ssstroke[void,arrowto]
 \ssgoto{h10^2b20z}
 \ssmove 0 1
 \ssstroke[void,arrowto]
 \ssgoto{x10h10^2b20z}
 \ssmove 0 1
 \ssstroke[void,arrowto]

\end{sseq}\end{equation}

The entire pattern described by the chart~\ref{ss e3} repeats: there is the periodicity class (not pictured) $b_{20}^2$ in bidegree $(s=4,u-s=8)$, which maps, under the 
map of spectral sequences of Proposition~\ref{ss comparison with supertrivial coeffs}, to the element in $\Ext^{4,12}_{A(1)}(\mathbb{F}_2,\mathbb{F}_2)$
which, in the Adams spectral sequence, detects the famous real Bott periodicity element in $\pi_8(ko)$.

The classes whose names are not implied by the lines representing various multiplications are as follows, and whose names were not already given in our description of the $E_2$-term, are as follows:
\begin{itemize}
\item the class in bidegree $(s=3,u-s=4)$ is $h_{10}b_{20}$, which we abbreviate as 
$w_4$,
\item the class in bidegree $(s=2,u-s=5)$ is $x_{10}b_{20}$, which we abbreviate as 
$w_5$,
\item and the class in bidegree $(s=3,u-s=9)$ is $zb_{20}$, which we abbreviate as
$w_9$.
\end{itemize}

Finally, we write $b$ for $b_{20}^2$, so that the spectral sequence's $E_3$-term 
is multiplicatively generated by elements:

\begin{tabular}{c|c|c|c|c}
\mbox{Class} & \mbox{Coh. degree $s$} & \mbox{Ab. degree $t$} & \mbox{Top. degree $u$} & \mbox{Adams degree $u-s$} \\
$x_{10}$ & 0 & 1 & 1 & 1 \\
  $x_6$ & 0 & 7 & 6 & 6 \\
$h_{10}$ & 1 & 1 & 1 & 0 \\
$h_{11}$ & 1 & 2 & 2 & 1 \\
    $z$ & 1 & 7 & 6 & 5 \\
  $w_4$ & 3 & 9 & 7 & 4 \\
  $w_5$ & 2 & 9 & 7 & 5 \\
  $w_9$ & 3 & 15 & 12 & 9 \\
$b$ & 4 & 16 & 12 & 8  .
\end{tabular}

In Proposition~\ref{ss comparison with supertrivial coeffs} we constructed a map from the HH-May spectral sequence to the classical May spectral sequence
computing $\Ext_{A(1)}^{*,*}(\mathbb{F}_2,\mathbb{F}_2)$. We now draw the $E_3\cong E_{\infty}$-term of that classical May spectral sequence,
using the same conventions as charts~\ref{ss e2} and~\ref{ss e3}, so that one can easily see the (surjective) map of spectral sequence $E_3$-terms:

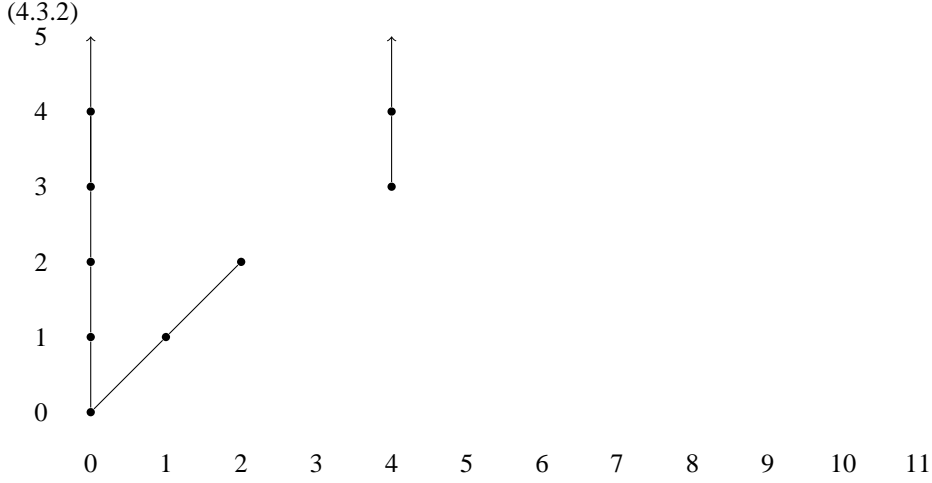
\begin{equation}
\begin{sseq}[grid=none,entrysize=10mm,labelstep=1]{0...11}{0...5}\label{classical may e3}
 \ssmoveto 0 0 
 \ssdropbull
 \ssname{1}

 \ssmove{0}{1} 
 \ssdropbull
 \ssname{h10}

 \ssmove 1 0
 \ssdropbull
 \ssname{h11}

 \ssmove{-1}{1}
 \ssdropbull
 \ssname{h10^2}

 \ssmove 2 0
 \ssdropbull
 \ssname{h11^2}

 \ssmove{-2}{1}
 \ssdropbull
 \ssname{h10^3}

 \ssmove 4 0
 \ssdropbull
 \ssname{h10b20}

 \ssmove{-4}{1}
 \ssdropbull
 \ssname{h10^4}

 \ssmove 4 0
 \ssdropbull
 \ssname{h10^2b20}
 
 \ssgoto{1}
 \ssgoto{h10}
 \ssstroke
 \ssgoto{1}
 \ssgoto{h11}
 \ssstroke

 \ssgoto{h10}
 \ssgoto{h10^2}
 \ssstroke

 \ssgoto{h11}
 \ssgoto{h11^2}
 \ssstroke

 \ssgoto{h10^2}
 \ssgoto{h10^3}
 \ssstroke

 \ssgoto{h10^3}
 \ssgoto{h10^4}
 \ssstroke 

 \ssgoto{h10^3}
 \ssgoto{h10^4}
 \ssstroke

 \ssgoto{h10b20}
 \ssgoto{h10^2b20}
 \ssstroke

 \ssgoto{h10^4}
 \ssmove 0 1
 \ssstroke[void,arrowto]

 \ssgoto{h10^2b20}
 \ssmove 0 1
 \ssstroke[void,arrowto]

\end{sseq}
\end{equation}

Again, the periodicity class (not pictured) is $b=b_{20}$ in bidegree $(s=4,u-s=8)$.

\begin{prop}\label{no more diffs}
In the abelianizing and the HH-May spectral sequences, all $d_r$ differentials are zero, for all $r>2$.
\end{prop}
\begin{proof}
We simply check that there can no nonzero $d_r$ differentials, for $r>2$, on the multiplicative generators
$x_{10},x_6,h_{10},h_{11},z,w_4,w_5,w_{9},b$ of the $E_3$-term of the abelianizing, equivalently (starting with $E_3$), the HH-May spectral sequence.
In the proof of Proposition~\ref{d2 differential computation}, we saw that $x_{10}, 
h_{10},h_{11},$ and $x_6$ all do not support nonzero differentials of any length
whatsoever: the first three by a cocycle-level calculation, and the last simply for degree reasons. The remaining generators of the $E_3$-term are all incapable of supporting nonzero $d_r$ differentials, for $r>2$, for degree reasons: there are no classes in the 
correct tridegree for any of these classes to hit by a $d_r$ differential, if $r>2$.
\end{proof}

\begin{theorem}\label{e-p series thm}
The spectral sequence chart~\ref{ss e3} displays (by reading across the rows) the Hochschild homology $HH_*(A(1),A(1))$.
In particular, the $\mathbb{F}_2$-vector space dimension of $HH_n(A(1),A(1))$ is:
\[ \dim_{\mathbb{F}_2} HH_n(A(1),A(1)) = \left\{ 
    \begin{array}{ll} 
      2n+5 &\mbox{\ if\ } 2\mid n \\
      2n+6 &\mbox{\ if\ } n\equiv 1 \mod 4 \\
      2n+4 &\mbox{\ if\ } n\equiv 3 \mod 4 .\end{array}\right. \]
Hence the Poincar\'{e} series of the graded $\mathbb{F}_2$-vector space $HH_*(A(1),A(1))$ is
\[ \frac{ 5 + 8s + 9s^2 + 10s^3 + \frac{8s^4}{1-s}}{1-s^4} .\]

If we additionally keep track of the extra grading on $HH_*(A(1),A(1))$ coming from the topological grading on $A(1)$, then 
the Poincar\'{e} series of the bigraded $\mathbb{F}_2$-vector space $HH_{*,*}(A(1),A(1))$ is
\begin{dmath}\label{e-p series} \left( (1+u)\left(1+u^2 + su(1+u^2+u^5) + s^2u^2(1+2u^5+u^7) + s^3u^3(1+u^4+u^5+u^6+u^9) + \frac{s^4u^4(1+u^4+u^5+u^9)}{1-su}\right) + u^6+su^2+su^8+s^2u^4\right)\frac{1}{1-s^4u^{12}}\end{dmath}
where $s$ indexes the homological grading and $u$ indexes the topological grading, as in Proposition~\ref{SSs and cocycle reps}.
\end{theorem}
\begin{proof}
This information is read off directly from the spectral sequence chart~\ref{ss e3}.
(Note that the horizontal axis in the chart~\ref{ss e3} is the Adams degree, i.e., $u-s$, not the internal/topological degree, i.e., $u$, so one must be a little careful
in reading off the series~\ref{e-p series} from the chart.)
\end{proof}

\section{The Hochschild cohomology of $A(1)$.}
\label{The Hochschild cohomology...}

Finite-dimensional Hopf algebras, such as $A(1)$, are Frobenius algebras. For a Frobenius algebra $A$ over a field $k$, there is a relatively straightforward and well-known\footnote{The author does not know where this duality originally appeared in the literature, but a nice account of the duality appears in \cite{MR3148613} and in section 3.1 of \cite{MR3421088}. See sections 1.1 and 1.2 of \cite{MR3198834} and section 3 of \cite{MR1388568} for good accounts of the most fundamental properties of graded Frobenius algebras. The duality does not seem to be a special case of van den Bergh's Poincar\'{e} duality for Hochschild (co)homology \cite{MR1443171}: as van den Bergh remarks in \cite{MR1900889}, the duality of \cite{MR1443171} requires the ring to be of finite Hochschild dimension, but the calculation of Hochschild homology we have just made in Theorem \ref{e-p series thm} demonstrates that $A(1)$ has infinite Hochschild dimension.} duality between Hochschild cohomology and Hochschild homology, as follows. Write $D(A)$ for the graded $k$-linear dual $\hom_k(A,k)$ of $A$: then $D(A)$ will be concentrated in non{\em positive} degrees, if $A$ is connected. The Frobenius form $\langle -,-\rangle$ on $A$ is the nondegenerate $k$-bilinear form given by fixing an isomorphism of the ground field $k$ with the (necessarily one-dimensional) highest-degree summand of $A$, and then letting $\langle x,y\rangle$ be the projection of the product $xy$ in $A$ to $k$. 

Since $A$ is Frobenius, it also admits a grading-preserving {\em Nakayama automorphism} $\nu: A\rightarrow A$ which is determined by the property that $\langle x,y\rangle = \langle y,\nu(x)\rangle$ for all homogeneous $x,y\in A$. Write $A_{\nu}$ for the $A$-bimodule whose underlying set of elements is $A$, whose left $A$-action is the usual left action of $A$ on itself (i.e., by left multiplication), and whose right $A$-action is given by letting $x\cdot a$ be $x\nu(a)$, i.e., $A$ acts on the right on $A_{\nu}$ by first applying $\nu$ and then multiplying on the right.

The relevant duality result is that $D(A) \cong \Sigma^{-n}A_{\nu}$ as graded $A$-bimodules, where $n$ is the highest degree of a nonzero element of $A$. (Sometimes this number $n$ is called the ``Artin-Schelter index of $A$,'' as in \cite{MR3198834}.) Consequently on the level of the Hochschild chain and cochain complexes $CC_{\bullet}$ and $CC^{\bullet}$, we have:
\begin{align*}
 D(CC_{\bullet}(A,A_{\nu})) 
  &\cong \hom_k(\BAR_{\bullet}(A)\otimes_{A^{\epsilon}} A_{\nu} ,k) \\
  &\cong \hom_{A^{\epsilon}}(\BAR_{\bullet}(A),\hom_k(A_{\nu},k)) \\
  &\cong \hom_{A^{\epsilon}}(\BAR_{\bullet}(A),D(\Sigma^n D(A))) \\
  &\cong CC^{\bullet}(A,\Sigma^{-n} A).
\end{align*}
The upshot is that we have an isomorphism
\begin{align}
\label{iso 340943} HH^i(A,A) &\cong \Sigma^n D\left(HH_i(A, A_{\nu})\right)\end{align}
of graded $k$-vector spaces for each integer $i$.

Using this duality in the case of $A(1)$, we get the following calculation of the Hochschild cohomology of $A(1)$:
\begin{theorem}\label{e-p series thm 2}
The $\mathbb{F}_2$-vector space dimension of the Hochschild cohomology\linebreak $HH^n(A(1),A(1))$ is equal to that of $HH_n(A(1),A(1))$, i.e., 
\[ \dim_{\mathbb{F}_2} HH^n(A(1),A(1)) = \left\{ 
    \begin{array}{ll} 
      2n+5 &\mbox{\ if\ } 2\mid n \\
      2n+6 &\mbox{\ if\ } n\equiv 1 \mod 4 \\
      2n+4 &\mbox{\ if\ } n\equiv 3 \mod 4 .\end{array}\right. \]
If we additionally keep track of the extra grading on $HH^*(A(1),A(1))$ coming from the topological grading on $A(1)$, then 
the Poincar\'{e} series of the bigraded $\mathbb{F}_2$-vector space $HH^{*,*}(A(1),A(1))$ is
\begin{dmath}\label{e-p series 2}
\left( u^{-7}(u+1)\left(u^{12}+u^{10} + su^6(u^5+u^3+1) + s^2u^3(u^7+2u^2+1) + s^3(u^9+u^5+u^4+u^3+1) + \frac{s^4(u^9+u^5+u^4+1)}{u-s}\right) + 1 +su^{4}+su^{-2}+s^2u^{2}\right)\frac{1}{1 - s^4u^{-12}}
\end{dmath}
where $s$ indexes the homological grading and $u$ indexes the topological grading.
\end{theorem}
\begin{proof}
By an elementary calculation, the Nakayama automorphism of $A(1)$ is the identity map, and the Artin-Schelter index of $A(1)$ (i.e., the degree of the socle of $A(1)$) is $6$. Consequently $A(1) = A(1)_{\nu}$, and the isomorphism \eqref{iso 340943} reduces to $HH^i(A(1),A(1)) \cong \Sigma^6 D\left(HH_i(A(1),A(1))\right)$. 

As a consequence, all that is necessary to obtain the Poincar\'{e} series of\linebreak $HH^*(A(1),A(1))$ is to replace $u$ with $u^{-1}$ in the Poincar\'{e} series \eqref{e-p series} for\linebreak $HH_*(A(1),A(1))$, and to multiply the result by $u^6$. This yields \eqref{e-p series 2}.
\end{proof}

Hochschild cohomology is of great use in classifying deformations. The most straightforward deformation-theoretic consequences of the calculation in Theorem \ref{e-p series thm 2} are as follows. Recall (e.g. from \cite{MR1383469}) that a {\em graded deformation} (respectively {\em $n$th order graded deformation} of an associative graded $k$-algebra $A$ is a graded $k$-algebra $A^{\prime}$ over the polynomial ring $k[t]$ (respectively, over the truncated polynomial ring $k[t]/(t^{n+1})$), with $t$ in degree $1$, such that $A^{\prime}$ is free over the module $k[t]$ (respectively, $k[t]/(t^{n+1})$), and
such that $A^{\prime}/tA^{\prime} \cong A$.
\begin{corollary}\label{uniqueness cor}
There are precisely four isomorphism classes of first-order graded deformations of $A(1)$. 

Suppose that $n>4$. If an $n$th order graded deformation of $A(1)$ extends to an $(n+1)$th order graded deformation of $A(1)$, then it does so uniquely, up to isomorphism.

Suppose furthermore that $n>7$. Then every $n$th order graded deformation of $A(1)$ extends to an $(n+1)$th order graded deformation of $A(1)$, unique up to isomorphism.
\end{corollary}
\begin{proof}
Reading off $HH^2$ and $HH^3$ from Theorem \ref{e-p series thm 2} or from the spectral sequence chart \eqref{classical may e3},
we have isomorphisms of graded $\mathbb{F}_2$-vector spaces.
\begin{align*}
 HH^2(A(1),A(1))
  &\cong \Sigma^{-4}\mathbb{F}_2\oplus \Sigma^{-3}\mathbb{F}_2\oplus \Sigma^{-2}\mathbb{F}_2\oplus \Sigma^{-2}\mathbb{F}_2 \\ 
  &\ \ \ \oplus \Sigma^{-1}\mathbb{F}_2\oplus \Sigma^{-1}\mathbb{F}_2\oplus \Sigma^2\mathbb{F}_2\oplus \Sigma^3\mathbb{F}_2\oplus \Sigma^4\mathbb{F}_2,\\
 HH^3(A(1),A(1))
  &\cong \Sigma^{-7}\mathbb{F}_2\oplus \Sigma^{-6}\mathbb{F}_2\oplus \Sigma^{-4}\mathbb{F}_2\oplus \Sigma^{-3}\mathbb{F}_2\oplus \Sigma^{-3}\mathbb{F}_2 \\ 
  &\ \ \ \oplus \Sigma^{-2}\mathbb{F}_2\oplus \Sigma^{-2}\mathbb{F}_2\oplus \Sigma^{-1}\mathbb{F}_2\oplus \Sigma^2\mathbb{F}_2\oplus \Sigma^3\mathbb{F}_2.
\end{align*}
The graded version of the infinitesimal deformation theory of associative algebras (see 1.5 of \cite{MR1383469}, or Proposition 5.4.1 of \cite{MR3971234} for a textbook account) establishes that the isomorphism classes of first-order infinitesimal deformations of a graded algebra $A$ are in bijection with $HH^{2,-1}(A,A)$, while the obstructions to extending an $(n-1)$th order infinitesimal deformation to an $n$th order infinitesimal deformation are cohomology classes in $HH^{3,-n}(A,A)$. When the extension of an $(n-1)$th order deformation to an $n$th order deformation is unobstructed, the set of isomorphism classes of such extensions is in bijection with $HH^{2,-n}(A,A)$.

In the case of $A(1)$, $HH^{2,-1}(A(1),A(1))$ is two-dimensional, hence has four elements. We have $HH^{2,-n}(A(1),A(1))\cong 0$ for all $n > 4$, and we have\linebreak $HH^{3,-n}(A(1),A(1))\cong 0$ for all $n > 7$. 
\end{proof}
More detailed consequences for deformations of $A(1)$ are possible using the Hochschild cohomology calculations in this paper: for example, one could calculate Gerstenhaber brackets in $HH^*(A(1), A(1))$ to determine {\em which} of the four first-order deformations extend to higher-order deformations. We regard such investigations as beyond the scope of this paper, though.


\begin{thebibliography}{10}

\bibitem{MR3519050}
Hossein Abbaspour.
\newblock On the {H}ochschild homology of open {F}robenius algebras.
\newblock {\em J. Noncommut. Geom.}, 10(2):709--743, 2016.

\bibitem{MR3148613}
Petter~Andreas Bergh and David~A. Jorgensen.
\newblock Tate-{H}ochschild homology and cohomology of {F}robenius algebras.
\newblock {\em J. Noncommut. Geom.}, 7(4):907--937, 2013.

\bibitem{MR1383469}
Alexander Braverman and Dennis Gaitsgory.
\newblock Poincar\'{e}-{B}irkhoff-{W}itt theorem for quadratic algebras of
  {K}oszul type.
\newblock {\em J. Algebra}, 181(2):315--328, 1996.

\bibitem{MR950556}
Jean-Luc Brylinski.
\newblock A differential complex for {P}oisson manifolds.
\newblock {\em J. Differential Geom.}, 28(1):93--114, 1988.

\bibitem{MR0814144}
Dan Burghelea.
\newblock The cyclic homology of the group rings.
\newblock {\em Comment. Math. Helv.}, 60(3):354--365, 1985.

\bibitem{cartiercohomologie}
Pierre Cartier.
\newblock Cohomologie des coalg\`{e}bres.
\newblock In {\em S\'{e}minaire ``{S}ophus {L}ie'' de la {F}acult\'{e} des
  {S}ciences de {P}aris, 1955-56. {H}yperalg\`ebres et groupes de {L}ie
  formels}, page~61. Secr\'{e}tariat Math\'{e}matique, 11 rue Pierre Curie,
  Paris, 1957.

\bibitem{MR3198834}
Kenneth Chan, Chelsea Walton, and James Zhang.
\newblock Hopf actions and {N}akayama automorphisms.
\newblock {\em J. Algebra}, 409:26--53, 2014.

\bibitem{MR597479}
Yukio Doi.
\newblock Homological coalgebra.
\newblock {\em J. Math. Soc. Japan}, 33(1):31--50, 1981.

\bibitem{MR3013261}
Bj{\o}rn~Ian Dundas, Thomas~G. Goodwillie, and Randy McCarthy.
\newblock {\em The local structure of algebraic {K}-theory}, volume~18 of {\em
  Algebra and Applications}.
\newblock Springer-Verlag London, Ltd., London, 2013.

\bibitem{MR972360}
S.~Geller, L.~Reid, and C.~Weibel.
\newblock The cyclic homology and {$K$}-theory of curves.
\newblock {\em J. Reine Angew. Math.}, 393:39--90, 1989.

\bibitem{MR3421088}
Thierry Lambre, Guodong Zhou, and Alexander Zimmermann.
\newblock The {H}ochschild cohomology ring of a {F}robenius algebra with
  semisimple {N}akayama automorphism is a {B}atalin-{V}ilkovisky algebra.
\newblock {\em J. Algebra}, 446:103--131, 2016.

\bibitem{MR1600246}
Jean-Louis Loday.
\newblock {\em Cyclic homology}, volume 301 of {\em Grundlehren der
  mathematischen Wissenschaften [Fundamental Principles of Mathematical
  Sciences]}.
\newblock Springer-Verlag, Berlin, second edition, 1998.
\newblock Appendix E by Mar\'{i}a O. Ronco, Chapter 13 by the author in
  collaboration with Teimuraz Pirashvili.

\bibitem{MR1474979}
Ib~Madsen.
\newblock Algebraic {$K$}-theory and traces.
\newblock In {\em Current developments in mathematics, 1995 ({C}ambridge,
  {MA})}, pages 191--321. Int. Press, Cambridge, MA, 1994.

\bibitem{MR0060829}
W.~S. Massey.
\newblock Products in exact couples.
\newblock {\em Ann. of Math. (2)}, 59:558--569, 1954.

\bibitem{MR2614527}
J.~Peter May.
\newblock {\em T{he} {cohomology} {of} {restricted} {Lie} {algebras} {and} {of}
  {Hopf} {algebras}: {application} {to} {the} {Steenrod} {algebra}}.
\newblock ProQuest LLC, Ann Arbor, MI, 1964.
\newblock Thesis (Ph.D.)--Princeton University.

\bibitem{MR1209233}
J.~E. McClure and R.~E. Staffeldt.
\newblock On the topological {H}ochschild homology of {$b{\rm u}$}. {I}.
\newblock {\em Amer. J. Math.}, 115(1):1--45, 1993.

\bibitem{MR0099653}
John Milnor.
\newblock The {S}teenrod algebra and its dual.
\newblock {\em Ann. of Math. (2)}, 67:150--171, 1958.

\bibitem{MR1181095}
Teimuraz Pirashvili and Friedhelm Waldhausen.
\newblock Mac {L}ane homology and topological {H}ochschild homology.
\newblock {\em J. Pure Appl. Algebra}, 82(1):81--98, 1992.

\bibitem{MR860042}
Douglas~C. Ravenel.
\newblock {\em Complex cobordism and stable homotopy groups of spheres}, volume
  121 of {\em Pure and Applied Mathematics}.
\newblock Academic Press Inc., Orlando, FL, 1986.

\bibitem{MR1388568}
S.~Paul Smith.
\newblock Some finite-dimensional algebras related to elliptic curves.
\newblock In {\em Representation theory of algebras and related topics
  ({M}exico {C}ity, 1994)}, volume~19 of {\em CMS Conf. Proc.}, pages 315--348.
  Amer. Math. Soc., Providence, RI, 1996.

\bibitem{MR0145525}
N.~E. Steenrod.
\newblock {\em Cohomology operations}.
\newblock Lectures by N. E. Steenrod written and revised by D. B. A. Epstein.
  Annals of Mathematics Studies, No. 50. Princeton University Press, Princeton,
  N.J., 1962.

\bibitem{MR1443171}
Michel van~den Bergh.
\newblock A relation between {H}ochschild homology and cohomology for
  {G}orenstein rings.
\newblock {\em Proc. Amer. Math. Soc.}, 126(5):1345--1348, 1998.

\bibitem{MR1900889}
Michel van~den Bergh.
\newblock Erratum to: ``{A} relation between {H}ochschild homology and
  cohomology for {G}orenstein rings'' [{P}roc. {A}mer. {M}ath. {S}oc. {\bf 126}
  (1998), no. 5, 1345--1348; {MR}1443171 (99m:16013)].
\newblock {\em Proc. Amer. Math. Soc.}, 130(9):2809--2810, 2002.

\bibitem{MR1269324}
Charles~A. Weibel.
\newblock {\em An introduction to homological algebra}, volume~38 of {\em
  Cambridge Studies in Advanced Mathematics}.
\newblock Cambridge University Press, Cambridge, 1994.

\bibitem{MR3971234}
Sarah~J. Witherspoon.
\newblock {\em Hochschild cohomology for algebras}, volume 204 of {\em Graduate
  Studies in Mathematics}.
\newblock American Mathematical Society, Providence, RI, [2019] \copyright
  2019.

\end{thebibliography}

\def\cprime{$'$} \def\cprime{$'$} \def\cprime{$'$} \def\cprime{$'$}

\end{document}